\documentclass{article}

\usepackage[text={6.2in, 7.9in},centering]{geometry}
\usepackage{amsmath,xcolor,cite}
\usepackage{bm,amssymb,mathrsfs,url,units,dsfont}
\usepackage{setspace,bbold}
\usepackage{graphicx,cite,mathtools}
\usepackage{amssymb,mathrsfs,paralist,bm,esint,setspace}
\RequirePackage{amsfonts,amsthm,amsmath,color}

\usepackage{amssymb}
\usepackage{esint,comment}
\usepackage{mathrsfs}
\usepackage{mathtools}
\usepackage{hyperref}
\usepackage{amsthm}
\usepackage{amsmath}
\usepackage{bm}
\usepackage{enumerate}
\usepackage{tikz-cd,soul}
\usepackage[normalem]{ulem}
\usepackage{bbold}
\usepackage{indentfirst}
\usepackage{titlesec}
\usepackage{footmisc}

\titlelabel{\thetitle.\,\,}

\numberwithin{equation}{section}
\theoremstyle{plain}

\newtheorem{theorem}{Theorem}[section]
\newtheorem{lemma}[theorem]{Lemma}

\newtheorem{remark}[theorem]{\bf{Remark}}
\newtheorem{assumption}[theorem]{Assumption}

\newtheorem{definition}[theorem]{Definition}

\theoremstyle{remark}
\theoremstyle{definition}
\makeatletter
\newcommand\footnoteref[1]{\protected@xdef\@thefnmark{\ref{#1}}\@footnotemark}
\makeatother

\newcommand\dP{\mathds{P}}
\newcommand\dR{\mathds{R}}
\newcommand\dE{\mathds{E}}

\newcommand\dN{\mathds{N}}
\newcommand\dH{\mathds{H}}
\newcommand\dL{\mathds{L}}
\newcommand\dD{\mathds{D}}
\newcommand\dS{\mathds{S}}

\newcommand\bi{\mathbb{1}}

\newcommand\cB{\mathcal{B}}

\newcommand\cF{\mathcal{F}}
\newcommand\cG{\mathcal{G}}
\newcommand\cH{\mathcal{H}}

\newcommand\cL{\mathcal{L}}
\newcommand\cP{\mathcal{P}}

\newcommand\cS{\mathcal{S}}

\newcommand\fm{\mathfrak{m}}
\newcommand\fA{\mathfrak{A}}
\newcommand\fB{\mathfrak{B}}

\newcommand\fp{\mathfrak{p}}
\newcommand\fq{\mathfrak{q}}
\newcommand\fr{\mathfrak{r}}

\newcommand\cbrk{\text{$]$\kern-.15em$]$}}
\newcommand\opar{\text{\,\raise.2ex\hbox{${\scriptstyle
|}$}\kern-.34em$($}}
\newcommand\cpar{\text{$)$\kern-.34em\raise.2ex\hbox{${\scriptstyle |}$}}\,}

\newcommand\ep{\varepsilon}

\begin{document}

\title{The compact support property for solutions to the stochastic partial differential  equations with colored noise  \footnote{Research supported in part by the National Research Foundation of Korea (NRF) grant NRF-2021R1C1C2007792 (B.H.) and NRF grants  2019R1A5A1028324 and 2020R1A2C4002077 [K.K. \& J.Y.]} }

\author{Beom-Seok Han
\and Kunwoo Kim  
\and Jaeyun Yi 
}

\date{\today}

\maketitle

\begin{abstract}  
We study the compact support property for nonnegative solutions of the following stochastic partial differential  equations: 
$$
\partial_t u =  a^{ij}u_{x^ix^j}(t,x)+b^{i}u_{x^i}(t,x)+cu\,+h(t,x,u(t,x))\dot{F}(t,x),\quad  (t,x)\in (0,\infty)\times\dR^d,
$$ where $\dot{F}$ is a spatially homogeneous Gaussian noise that is white in time and colored in space, and $h(t, x, u)$  satisfies $K^{-1}u^\lambda\leq h(t, x, u) \leq K(1+u)$ for  $\lambda\in(0,1)$ and  $K\geq 1$. We show that if the initial data $u_0\geq 0$ has a compact support, then,  under the \emph{reinforced} Dalang's condition on $\dot F$ (which guarantees the existence and the H\"older continuity of  a weak solution), all nonnegative weak solutions $u(t, \cdot)$ have  the compact support  for all $t>0$ with probability 1.  Our results extend  the works by Mueller-Perkins (\cite{mueller1992compact}) and Krylov (\cite{krylov1997result}),  in which they show the compact support property only for the one-dimensional   SPDEs driven by   space-time white  noise on $(0, \infty)\times \dR$.

\vspace{1cm} 
 
	
\end{abstract}



\section{Introduction}\label{sec: introduction}
In this paper, we aim to study the compact support property of nonnegative solution to the following stochastic partial differential equations
 \begin{equation}
\label{original_equation}
\partial_t u(\omega,t,x) = \cL u(\omega,t,x)\,+h(\omega,t,x,u(t,x))\dot{F}(t,x),\quad  (t,x)\in (0,\infty)\times\dR^d;\quad u(0,\cdot) = u_0,
\end{equation}
where $\cL$ is the random differential operator defined as 
\begin{equation*}
\begin{gathered}
\cL u(\omega,t,x)  = a^{ij}(\omega,t,x) u_{x^ix^j} + b^i(\omega,t,x) u_{x^i} + c(\omega,t,x) u
\end{gathered}
\end{equation*}
and $h(\omega,t,x,u)$ is a function such that for some nonrandom constants  $K\geq 1$ and $\lambda\in(0,1)$,
\begin{equation}\label{eq:condition_h}
K^{-1}u^\lambda \leq h(\omega,t,x,u)  \leq K(1+u) 
\quad\quad \text{for all \quad\quad$\omega\in\Omega,\,\,t>0,\,\,x\in\dR^d,\,\,u\geq0$}.
\end{equation}
We assume Einstein's  summation convention on $i$ and $j$, and $i$ and $j$ go from 1 to $d$. 
The detailed conditions on $a,b,c$, and $h$ are specified below in Assumptions \ref{assumptions on coefficients} and \ref{assumptions on h}. 
The nonrandom initial data $u_0$ is a nonnegative H\"older continuous function with compact support. The noise $\dot{F}$ is a centered generalized  Gaussian random field whose covariance is given by
\begin{equation*}
\dE[F(t,x)F(s,y)] = \delta_0(t-s)f(x-y),
\end{equation*}
where $\delta_0$ is the  Dirac delta distribution and $f$ is a {\it correlation function (or measure)}. In other words, $f$ is a nonnegative and nonnegative definite function/measure. We denote that the spectral measure  $\mu$ of $f$ is a nonnegative tempered measure defined by
\begin{equation}
 	\label{spectral measure}
 	\mu(\xi):=\cF(f)(\xi) := \frac{1}{(2\pi)^{d/2}}\int_{\dR^d} e^{-i\xi\cdot x } f(dx)
\end{equation} 
which is the Fourier transform of $f$.

The {\it compact support property (CSP)} for \eqref{original_equation} means that if the initial function $u_0$ has a compact support, then so does the solution $u(t, \cdot)$ for all $t>0$ with probability 1, and this property has been studied extensively for one-dimensional stochastic heat equations driven by space-time white noise.  More precisely, consider the following stochastic heat equations:
\begin{equation}\label{eq:SHE} 
\partial_t u = u_{xx} + u^\lambda \, \dot{W}, \quad (t,x)\in (0,\infty)\times\dR;\quad u(0,\cdot) = u_0,
\end{equation}
where $\dot W$ is space-time white noise (i.e. $f=\delta_0$) and $\lambda>0$ is a fixed constant. 

When $\lambda=1/2$ in \eqref{eq:SHE}, the solution to \eqref{eq:SHE} can be regarded as a density of super-Brownian motion, and Iscoe (\cite{Iscoe1988}) showed that CSP holds for \eqref{eq:SHE} by using some singular elliptic boundary value problem  related to properties of super-Brownian motion.  Shiga (\cite{shiga1994two}) extended Iscoe's technique to show CSP for \eqref{eq:SHE} when $ \lambda \in (0\,, 1/2)$. For $\lambda \in (1/2\,, 1)$, Mueller and Perkins (\cite{mueller1992compact}) showed among other things that  CSP holds for \eqref{eq:SHE}  by constructing and then applying  a {\it  historical process} that represents the ancestry of particles. In fact, the case where $\lambda \in (1/2, 1)$ is more delicate, as explained in \cite{mueller1992compact}, since this case can be regarded as a super-Brownian motion with birth and death rates slowed down (that is, the coefficient of the noise $u^\lambda$ is  smaller than $u^{1/2}$  when $u$ is close to 0), but still the effect of noise is strong enough for \eqref{eq:SHE} to have CSP.  On the other hand, when $\lambda\geq 1$, Mueller (\cite{Mueller1991}) showed that CSP does not hold. In this case, when $u$ is close to 0, $u^\lambda$ for $\lambda\geq 1$ becomes quite   small so that the positivity of the heat semigroup wins against the effect of noise.  Hence, for \eqref{eq:SHE}, the critical exponent $\lambda$  that does not guarantee CSP is $1$. 

Regarding the one-dimensional SPDE \eqref{original_equation} with the random second order differential operator $\cL$ and with  space-time white noise, using completely different methods that are based on the weak solution to \eqref{original_equation}, Krylov (\cite{krylov1997result}) showed that CSP holds when $h(u) = u^\lambda$ for any $\lambda \in (0\, , 1)$ under the hypothesis of the existence of a suitable solution. Following the same argument by Krylov (\cite{krylov1997result}), Rippl ( \cite{rippl2013pathwise}) showed that CSP holds for  \eqref{eq:SHE}  when $\dot W$ is replaced by $\dot F$ whose spatial covariance is of Riesz type, i.e., $f(x)=|x|^{-\alpha}$ for $\alpha \in (0, 1)$ and $x\in \dR$. The proof therein strongly depends on the scaling property of the specific correlation function $|x|^{-\alpha}$. Moreover, both authors in \cite{krylov1997result} and \cite{rippl2013pathwise} only considered the one-dimensional SPDEs, and the generalization of their techniques to higher dimensions is highly nontrivial, which is one of our main contributions in this paper.

When $d\geq 2$, in order to have a function-valued solution to \eqref{original_equation}, noise should be colored in space. In particular, when $\cL=\Delta$ and $h(u)$ is globally Lipschitz, Dalang (\cite{dalang1999extending}) showed that  if  the spectral measure $\mu$ of $f$ satisfies 
\begin{equation}\label{Dalang's condition0}
	\int_{\dR^d} \frac{\mu(d \xi)}{1+|\xi|^2} <  \infty, 
\end{equation} 
there exists a unique random field solution.  However, this condition itself does not guarantee that the solution is  continuous in $t$ and $x$ (see \cite{CHKK2019,FK2013}), and  one sufficient condition that guarantees the H\"older continuity of the solution  is  the following:  
\begin{equation}\label{Dalang's condition}
	\int_{\dR^d} \frac{\mu(d \xi)}{(1+|\xi|^2)^{1-\eta}} < +\infty \quad \text{for some $\eta\in(0,1]$,}
\end{equation} 
which is called the {\it reinforced Dalang's condition} (see \cite{sanz2002holder,ferrante2006spdes,CH2019}). Under the condition \eqref{Dalang's condition}, when $\cL = \Delta$  and $h(u) \approx u$ in \eqref{original_equation}, following Mueller's argument (\cite{Mueller1991}),  Chen and Huang (\cite{CH2019}) showed that CSP does not hold. In addition, when $\cL = \Delta$ and $h(u)=u^\lambda$ for $\lambda>1$ in \eqref{original_equation}, following Mueller's argument (\cite{Mueller1991}),  one can show  that  CSP does not hold under the condition \eqref{Dalang's condition}.  Thus, a natural question arises whether CSP holds for \eqref{original_equation} when $h$ satisfies \eqref{eq:condition_h} for $\lambda \in (0\, ,1)$.  To the best of our knowledge, there is no literature on CSP for \eqref{original_equation} even when $\cL=\Delta$ on $\dR^d$ for $d\geq 2$,  and we provide an answer in the affirmative in this paper. 

Our main result of this paper is that when $h$ satisfies \eqref{eq:condition_h} and the random second order differential operator $\cL$ in \eqref{original_equation} satisfies some natural condition (see Assumption \ref{assumptions on coefficients}),  CSP holds for \eqref{original_equation} only under the condition \eqref{Dalang's condition},  which includes $f(x)=\delta_0(x)$, $f(x)=|x|^{-\alpha}$ for $\alpha \in (0, 2\wedge d)$, $f\in C_c(\dR^d)$, and $f(x)=1$ for $x\in \dR^d$. Note that $f\equiv 1$ means  the noise is just white noise in time only, i.e., $\dot F(t, x)=\dot  B_t$ where $B_t$ is the standard one-dimensional Brownian motion. Therefore, our result highlights that CSP is  essentially due to the sub-linearity $u^\lambda$ with $\lambda \in (0\,,1)$, not to the choice of noise. In addition, it also suggests that, as for one-dimensional stochastic heat equations with space-time white noise,  the critical exponent $\lambda$ that does not guarantee CSP for \eqref{original_equation} is 1.

Let us now mention  a few words about the proof. Since \eqref{original_equation} includes the random second order differential operator $\mathcal{L}$, we  use the approach given by Krylov \cite{krylov1997result}. Krylov \cite{krylov1997result} assumed the existence of   a nonnegative  solution. From this assumption, he obtained Lemma 2.1 in \cite{krylov1997result}, which is the crucial lemma that basically  produces CSP.  In this paper, on the other hand,  we first show by using a compactness argument and the $L_p$ theory (\cite{kry1999analytic, kry1996})  that a  solution exists and is also nonnegative  (see Theorem \ref{thm:existence of solution}), and then obtain Lemma \ref{ito_prop} that corresponds to Lemma 2.1 in \cite{krylov1997result}. However, since our spatial space can be $\dR^d$ for $d\geq 2$, which makes  the boundary of a ball in $\dR^d$ much more complex than the boundary of a ball in $\dR$ (i.e., just single points), we need another lemma (Lemma \ref{L_1 bound with weight})  for the proof of CSP in order to control the boundary behavior of the solution. In addition, because our noise is spatially correlated,  some new ideas should be developed  to obtain Lemma \ref{ito_prop}. The first idea is to use some comparison argument that can compare the given nonnegative definite measure $f$ with another nonnegative definite function that can be written as a convolution of some other functions (see Lemma \ref{lemma:lower bound of covariance}). This comparison allows us to bound  the integral of a function by the integral of a function convoluted with the nonnegative definite measure $f$ (see \eqref{ineq:key ineq in the proof of key lemma 4}).  Another one is Lemma \ref{lemma:integral inequality}, which is used to adjust the exponent of integrals in Lemma \ref{ito_prop}. In order to use Lemma \ref{lemma:integral inequality}, the solution $u$ should be H\"older continuous in $t$ and $x$, and we show H\"older continuity of the solution by using the $L_p$-theory (see \cite{kry1999analytic, kry1996}). 
  
The paper is organized as follows. In Section \ref{sec:Main}, we provide some basic definitions and assumptions, and then state our main theorems (Theorems \ref{thm:existence of solution} and \ref{thm:cpt_support}). The proof of Theorem \ref{thm:existence of solution} is given in Section \ref{sec:proof of holder conti of sol}, whereas the proof of Theorem \ref{thm:cpt_support} is provided     in Section \ref{sec:proof of main thm} based on the auxiliary results that are proved in Sections \ref{subsec:proof of props} and \ref{sec:proof of L_1 bound}.

We finish this section with some notations. Throughout this paper, we assume Einstein’s summation convention on $i$, $j$, and $k$. The sets $\dN$ and $\dR$ are the sets of natural numbers and real numbers, respectively. The set $\dR^d$ denotes the $d$-dimensional Euclidean space of points $x = (x^1,\dots,x^d)$ for $x^i\in\dR$. The set $\cS: = \{ f\in C^\infty(\dR^d) : \sup_{x\in \dR^d} | x^\alpha (D^\beta f)(x)| <\infty \text{ for all $\alpha,\beta \in \dN^d$} \}$ denotes the set of Schwartz functions on $\dR^d$.  We denote a generic constant by $N = N(a_1,a_2,...,a_k)$ if $N$ depends only on $a_1,a_2,...,a_k$. The numeric value of $N$ can be changed line by line tacitly. 
For functions  depending on $\omega$, $t$, and $x$, the argument $\omega \in \Omega$ is omitted. Finally, for a function $\varphi(x),$ let $\tilde{\varphi}(x)$ and $\hat{\varphi}(\xi)$ denote $\varphi(-x)$ and $\cF(\varphi)(\xi)$, respectively.


\section{Main Results}\label{sec:Main}
In this section, we first provide  some basic definitions and assumptions related to \eqref{original_equation}, and then state our main theorems (Theorems \ref{thm:existence of solution} and \ref{thm:cpt_support}).

Let $(\Omega, \cF, \dP)$ be a complete probability space,   $\{\cF_t : t\geq0 \}$  be an increasing  filtration  of $\sigma$-fields $\cF_t \subset \cF$ satisfying the usual conditions, and  $\cP$ be the predictable $\sigma$-field related to $\cF_t$. Consider a mean-zero Gaussian process $\{F(\phi):\phi\in\cS(\dR^{d+1})\}$ on $(\Omega, \cF, \dP)$ with covariance functional given by 
\begin{equation}
\label{def of covaiance functional}
\begin{aligned}
\dE \left[F(\phi)F(\psi)\right] 
&= \int_0^\infty  dt \int_{\dR^d}  f(dx)\, (\phi(t,\cdot) *  \tilde \psi(t, \cdot))(x)  \\
& = \int_0^\infty dt  \int_{\dR^d} \mu(d\xi)\,  \hat \phi(t, \xi) \, \overline{\hat \psi(t, \xi)}, 
\end{aligned}
\end{equation}
where $\tilde \psi (t, x):=\psi(t, -x)$, $\hat\phi(t, \cdot):= \cF(\phi(t,\cdot)) (\xi)$ is the Fourier transform of $\phi$ with respect to $x$,  $f$ is a nonnegative, and nonnegative definite tempered  measure, and $\mu$ is the Fourier transform of $f$ introduced in \eqref{spectral measure} so that it is again  a nonnegative definite tempered measure. Following the Dalang-Walsh theory (see \cite{walsh1986introduction,dalang1999extending}), the process $F$ can be extended to an $L^2(\Omega)$-valued martingale measure $F(ds,dx)$, where $F(ds,dx)$ is employed to formulate \eqref{original_equation} in the weak sense (see Definition \ref{definition_of_solution} and Remark \ref{random measure}). We impose the following assumption on $f$.

\begin{assumption}\label{assumption on f}
The correlation measure $f$ satisfies  the following: For some $\eta\in(0,1]$,
\begin{equation}
\label{reinforced Dalang's condition}
\begin{aligned}
\begin{cases}
\int_{|x|<1} |x|^{2-2\eta -d} \, f(dx)  < +\infty \quad &\text{if}\quad 0<1-\eta<\frac{d}{2},	\\
\int_{|x|<1} \log\left(\frac{1}{|x|}\right)\, f(dx) < +\infty \quad &\text{if}\quad 1-\eta=d/2, \\
\text{no conditions on }f &\text{if}\quad  1-\eta>d/2.
\end{cases}
\end{aligned}
\end{equation} 

\end{assumption}

\begin{remark} \begin{enumerate}[(i)]

\item   \eqref{reinforced Dalang's condition} is equivalent to the reinforced Dalang's condition \eqref{Dalang's condition} (see \cite[Proposition 5.3]{sanzsolesarra2000}).  As explained before,  the reinforced Dalang's condition  guarantees that a solution exists and is H\"older  continuous in $t$ and $x$ when $h$ is globally Lipschitz in $u$ in \eqref{original_equation} (see \cite[Theorem 6]{ferrante2006spdes} or \cite[Theorem 3.5]{choi2021regularity}).

\item 
A large class of nonnegative, nonnegative definite tempered measures  satisfies Assumption \ref{assumption on f}. For example, one may consider $f(dx) = f(x)dx$, where $f(x)$ can be the Riesz kernel $f(x)= |x|^{-\alpha}$ with $\alpha\in (0,2\wedge d)$, the Ornstein-Uhlenbeck-type kernel $f(x)= \exp(-|x|^\beta)$ with $\beta\in(0,2]$, and the Brownian motion $f(x) \equiv 1$. Additionally, a continuous, nonnegative and nonnegative definite function $f$ with compact support satisfies \eqref{reinforced Dalang's condition}.

\end{enumerate}
\end{remark}

Below Assumptions \ref{assumptions on coefficients} and \ref{assumptions on h} describe conditions on coefficients. 

\begin{assumption}
\label{assumptions on coefficients}
\begin{enumerate}[(i)]
\item 
The coefficients $a^{ij} = a^{ij}(t,x)$, $b^i = b^i(t,x)$ and $c = c(t,x)$ are $\cP\times\cB(\dR^d)$-measurable. 

\item 
The functions $a^{ij}_{x^ix^j}$, $a^i_{x^i}$, $a^i$, $b^i_{x^i}$, $b^i$, and $c$ are continuous in $x$.

\item
There exists a finite constant $K\geq 1$ such that
\begin{equation}
\label{ellipticity}
K^{-1}|\xi|^2\leq  a^{ij}\xi^i \xi^j\leq K|\xi|^2\quad \mbox{for all}\quad \omega\in \Omega,\,\,t>0,\,\,x\in\dR^d,\,\,\xi\in\dR^d
\end{equation}
and
\begin{equation*}
|a^{ij}(t,\cdot)|_{C^2{(\dR^d)}} +  |b^{i}(t,\cdot)|_{C^1{(\dR^d)}} + |c(t,\cdot)|_{C{(\dR^d)}} \leq K\quad\mbox{for all}\quad \omega\in\Omega,\,\,t>0,
\end{equation*} 
where $C^k(\dR^d)$ $(k=1,2)$ is the set of all $k$ times continuously differentiable functions on $\dR^d$ with finite norm 
\begin{equation*}
	\| u \|_{C^k(\dR^d)} := \sum_{|\alpha|\leq k}\sup_{x\in\dR^d}|D^\alpha u(x)|.
\end{equation*}

\end{enumerate}
\end{assumption}

\begin{assumption}
\label{assumptions on h}
Let $\lambda\in (0,1)$.
\begin{enumerate}[(i)]
\item 
The function $h(t,x,u)$ is $\cP\times\cB(\dR^d)\times\cB(\dR)$-measurable and $h(t,x,\cdot):\dR \rightarrow \dR$ is continuous for every $\omega \in \Omega, t>0$ and $x\in\dR^d$.  

\item The function $h$ satisfies 
\begin{equation}
	h(t,x,u) = 0 \quad\text{for all}\quad   \omega\in\Omega,\,\, t>0,\,\, x\in\dR^d,\,\, u \leq 0.
\end{equation}

\item 
There exists a finite constant $K\geq 1$ such that
\begin{equation}
\label{condition of h} K^{-1} u^\lambda \leq h(t, x, u)  
\leq K ( 1 + u ) 
\quad \text{for all}\quad \omega\in\Omega,\,\, t>0,\,\, x\in\dR^d,\,\, u \geq 0.
\end{equation}

\end{enumerate}
\end{assumption}

\begin{remark}

\begin{enumerate}[(i)]

\item 
The coefficients $a^{ij}$, $b^i$, $c$, and $h$ are random functions.

\item Without loss of generality, we assume the same constant $K$ in Assumptions \ref{assumptions on coefficients} and \ref{assumptions on h}.

\item 
In \eqref{condition of h}, the upper bound of $h$ is used to obtain the existence and the H\"older regularity of a solution. On the other hand, the lower bound of $h$ is employed to calculate the probability estimate of the integral of the solution on circles $\{ x:|x| = R \}$ and $\{ x:|x| = R + r \}$ $(R,r>0)$, which is essential to obtain the compact support property of a solution $u$; see Lemma \ref{ito_prop} (ii).

\end{enumerate}
\end{remark}

Next we introduce the assumption on the initial data $u_0$. 

\begin{assumption}
\label{assumption on initial data}
\begin{enumerate}[(i)]
\item 
The function $u_0(\cdot)$ is nonrandom, nonnegative, and $u_0(\cdot)\in C^\eta(\dR^d)$, where $\eta$ is the constant introduced in \eqref{reinforced Dalang's condition}.

\item 
The function $u_0(\cdot)$ has a compact support on $\dR^d$. In other words, there exists $R_0\in(0,\infty)$ such that support of $u_0$ is in $\{ x\in\dR^d: |x|<R_0 \}$.
\end{enumerate}
\end{assumption}

We now present the definition of a solution. First, we define
$$C_{tem} := \left\{ u\in C(\dR^d):\sup_{x\in\dR^d}|u(x)|e^{-a |x|}<\infty\text{ for any } a>0 \right\}.
$$ 

In the rest of this section, $\tau$ is any given bounded stopping time, i.e., $\tau\leq  T$ for some nonrandom number $T>0$ almost surely. 

\begin{definition} 
\label{definition_of_solution}
A continuous $C_{tem}$-valued function $u = u(t,\cdot)$ defined on $\Omega\times[0,\tau]$ is said to be a solution to \eqref{original_equation}  on $[0,\tau]$ if for any $\phi\in \cS$
\begin{enumerate}[(i)]
\item 
the process $\int_{\dR^d} \phi(x)u(t\wedge \tau,x) dx$ is well-defined, 
$\cF_t$-adapted, and continuous;

\item 
the process 
$$ \int_{\dR^d}\int_{\dR^d} h(t\wedge \tau,y,u(t \wedge \tau,y))\phi(y) h(t \wedge \tau,y+x,u(t\wedge \tau,y+x))\phi(y+x) f(dx)dy$$ 
is well-defined, $\cF_t$-adapted, measurable with respect to $(\omega,t)$, and 
\begin{equation}
\label{quadratic variation of stochastic part}
\begin{gathered}
 \int_0^\tau \int_{\dR^d}\int_{\dR^d}h(t,y,u(t,y))\phi(y)h(t,y+x,u(t,y+x))\phi(y+x)f(dx) dy dt <\infty
\end{gathered}
\end{equation}
almost surely;

\item
the equation
\begin{equation} 
\label{eq:sol_int_eq_form}
(u(t,\cdot),\phi) = (u_0,\phi)+\int_0^t(u(s,\cdot), (\cL^*\phi)(s,\cdot)) ds + \int_0^t\int_{\dR^d}h(s,x,u(s,x))\phi(x) F(ds,dx)
\end{equation}
holds for all $t\leq \tau$ almost surely, where
\begin{equation}\label{eq:adjoint}
\begin{aligned}
& (\cL^* \phi)(s,x) = \left( a^{ij}_{x^ix^j} -  b^i_{x^i} + c \right)\phi  + \left( 2 a^{ij}_{x^j} - b^i \right) \phi_{x^i} +  a^{ij}  \phi_{x^ix^j}.
\end{aligned}
\end{equation}

\end{enumerate}

\end{definition}

\begin{remark}
\label{random measure}
\begin{enumerate}[(i)]
\item 
\label{random measure-1}
In \eqref{eq:sol_int_eq_form}, the Dalang-Walsh theory  is employed   to define a stochastic integral with respect to $F(ds,dx)$ and we call it Walsh's stochastic integral; see \cite{dalang1998stochastic,dalang1999extending}. As in \cite{ferrante2006spdes,choi2021regularity}, Walsh's stochastic integral can be written as an infinite summation of It\^o's stochastic integral; for example, if $X(t,\cdot)$ is a predictable process such that
$$ X(t,x) = \zeta(x)1_{ \opar\tau_1,\tau_2\cbrk}(t),
$$
where $\tau_1$, $\tau_2$ are bounded stopping times, $\opar\tau_1,\tau_2\cbrk:=\{ (\omega,t):\tau_1(\omega)<t\leq \tau_2(\omega) \}$, and $\zeta\in C_c^\infty$, then we have 
\begin{equation} 
\label{decomposition of walsh integral}
\int_{0}^t\int_{\dR^d}X(s,x)F(ds,dx)=\sum_{k=1}^{\infty}\int_0^t \int_{\dR^d} X(s,x) \, (f\ast e_k)(x) \, dxdw^k_s,
\end{equation}
where $\{w_t^k, k \in \dN\}$ is a collection of  one-dimensional independent Wiener processes and $\{ e_k, k \in \dN\}\subseteq \cS$ is a complete orthonormal system of a Hilbert space $\cH$ induced by $f$; see \cite{ferrante2006spdes,choi2021regularity}. The construction and properties of $\cH$ and $\{e_k, k\in\dN \}$ are described in \cite[Remark 2.6]{choi2021regularity}.

\item 
Thanks to \eqref{quadratic variation of stochastic part} and \eqref{decomposition of walsh integral}, the stochastic integral term in \eqref{eq:sol_int_eq_form} can be written as 
\begin{equation*}
\begin{aligned}
\int_0^t\int_{\dR^d}h(s,x,u(s,x))\phi(x) F(ds,dx)= \sum_{k=1}^{\infty}\int_0^t \int_{\dR^d}h(s,x,u(s,x))\phi(x)   \, (f\ast e_k)(x) \, dxdw^k_s.
\end{aligned}
\end{equation*}

\end{enumerate}
\end{remark}

Throughout this paper, we always assume that Assumptions \ref{assumption on f}, \ref{assumptions on coefficients}, \ref{assumptions on h}, and  \ref{assumption on initial data} hold. Now we show  the  existence and H\"older regularity of a solution. Since the proof of the following theorem is  rather classical, we include the proof at  the end (see Section \ref{sec:proof of holder conti of sol}).    

\begin{theorem}[\bf Weak existence and H\"older regularity of the solution]
\label{thm:existence of solution} 
There exists a stochastically weak solution $u \in C([0,\tau];C_{tem}(\dR^d))$ to \eqref{original_equation} satisfying $u\geq0$. In addition, for any solution $u \in C([0,\tau];C_{tem}(\dR^d))$ to \eqref{original_equation} and any $\gamma\in\left(0,\eta/2\right)$, we have that for every $a>0$,
\begin{equation}
\label{holder regularity of the solution}
\| \Psi_{a} u \|_{C^{\gamma}([0,\tau]\times\dR^d)} <\infty\quad\text{almost surely,} 
\end{equation}
where $\eta$ is the constant introduced in \eqref{reinforced Dalang's condition} and $\Psi_a(x) = \Psi_a(|x|) = \frac{1}{\cosh(a|x|)}$.

\end{theorem}

\begin{remark}
When $\mathcal{L}=\Delta$ in \eqref{original_equation},  
Mytnik-Perkins-Sturm (\cite{mytnik2006pathwise}) showed the existence of a stochastically weak solution by using a compactness argument.  Indeed, they also showed the pathwise uniqueness of  a solution under certain conditions on $h$ and $f$. Here, we also use a compactness argument but in a different way. That is, we use the $L_p$-theory  to get tightness conditions (see  Section \ref{sec:proof of holder conti of sol}). However, since we have the random second order differential operator in \eqref{original_equation}, it is not easy to show the pathwise uniqueness by  applying the arguments in \cite{mytnik2006pathwise}. Thus, we do not assume  the solution is unique, but this does not affect our proof since we only care about the compact support property of any solution.  
\end{remark}

Now we introduce the main result of this paper, which states the compact support property holds for \eqref{original_equation}.

\begin{theorem}[\bf Compact support property]
\label{thm:cpt_support}
Suppose $u$ is a nonnegative solution to \eqref{original_equation} and $R_0$ is the constant introduced in Assumption \ref{assumption on initial data}. Then, for almost every $\omega$, there exists $R=R(\omega)\in(R_0,\infty)$ such that $u(t,x) = 0$ for all $t\leq \tau$ and $x\in \{ x\in \dR^d: |x| > R \}$.
\end{theorem}


\section{Proof of Theorem \ref{thm:cpt_support}}
\label{sec:proof of main thm}

 In this section, we assume the results of Theorem \ref{thm:existence of solution} and provide the proof of Theorem \ref{thm:cpt_support}. Our proof  is based on the approach of Krylov  \cite{krylov1997result}. We first present two lemmas (Lemmas \ref{L_1 bound with weight} and \ref{ito_prop}), which are  
essential tools for the proof of Theorem  \ref{thm:cpt_support} and whose proofs are given in Sections \ref{sec:proof of L_1 bound} and \ref{subsec:proof of props} respectively.  Throughout this section, $u$ is referred to any solution to \eqref{original_equation} that is  nonnegative and continuous almost surely. We also use the following notations through  the remaining part of the paper: For $a>0$, define
\begin{equation}
\label{def of Psi}
	\Psi_a(x) := \Psi_a(|x|) := \frac{1}{\cosh(a|x|)}.
\end{equation}
For $r>0$, set
$$Q_{r} := \left\{ x  \in \dR^d : |x|>r  \right\}.
$$
The measure $d\sigma_R$ denotes surface measure on $\partial Q_R$. Let $T>0$ be an arbitrary fixed real constant and $R_0>0$ be the constant introduced in Assumption \ref{assumption on initial data},  i.e., $supp(u_0) \subseteq Q_{R_0}^c$. The following lemma shows the limit behavior of $u$ on $\partial Q_R$ as $R\to\infty$.

\begin{lemma}
\label{L_1 bound with weight} 
Let $\tau\leq T$ be a bounded stopping time. Suppose that there exist $a,H\in (0,\infty)$ 
such that 
\begin{equation}
\label{holder reg for ito ineq0}
\sup_{t\leq\tau}\sup_{x\in\dR^d} \Psi_a(x)u(t,x)\leq H \quad \text{almost surely},
\end{equation}
where $\Psi_a$ is the function introduced in \eqref{def of Psi}. Then, for every $\delta >0$, we have
\begin{equation}
\label{eq:integral bound with weight}
\limsup_{R\to\infty}\dE\left[\int_0^\tau e^{\delta R}\int_{\partial Q_R} u(t,\sigma) d\sigma_R  dt\right]= 0.
\end{equation}

\end{lemma}

\vspace{2mm}

In Lemma \ref{ito_prop} below,  a version of the maximum principle for $u$ is introduced. In other words, we show that if $u$ is zero on $\partial Q_R$ almost surely, $u$ is also zero on $Q_R$.  The last part of the lemma, which has a crucial role in the proof of Theorem \ref{thm:cpt_support}, shows a probability estimate of the integral value of the solution on circles.

\begin{lemma}
\label{ito_prop}
Let $\tau\leq T$ be a bounded stopping time and $K$ be the constant introduced in Assumptions \ref{assumptions on coefficients} and \ref{assumptions on h}. Set $ l := \frac{\gamma\lambda+d}{\gamma+d}\in(0,1)$, where $\gamma$ is the constant introduced in Theorem \ref{thm:existence of solution}. Suppose that there exist $a,H\in(0,\infty)$ such that 
\begin{equation}
\label{holder reg for ito ineq}
\|\Psi_a u\|_{C^\gamma([0,\tau]\times\dR^d)}\leq H \quad \text{almost surely},
\end{equation} 
where $\Psi_a$ is the function introduced in \eqref{def of Psi}. Let $R_0$ be the constant introduced in Assumption \ref{assumption on initial data}. Then, we have the following:
\begin{enumerate}[(i)]
\item \label{key_prop_1}
For $R>R_0\vee1$,
\begin{equation*}
    \dP \left(  \int_0^\tau \left( \int_{\partial Q_R}  u(s,\sigma) d\sigma_R \right)^ l ds = 0 \right) \leq \dP \Big( u(s,x) =0 \quad \text{for all }s\in[0,\tau] \text{ and } x\in Q_R\Big).
\end{equation*}

\item \label{key_prop_2}
For $R>R_0 \vee1$, $\fp,\fq>0$, 
\begin{equation}
\label{condition of r for decay prop}
0 < \fr <  \frac{R}{2}\wedge \left[  \left(2^{-\frac{(\gamma+1)}{\gamma}}H^{-\frac{d}{\gamma}} R^{d-1} \left( \frac{d\pi^{d/2}}{\Gamma(d/2+1)} \right) \right)^{\frac{1}{\gamma+d-1}}\right] \wedge 1,
\end{equation}
and $\alpha\in (0,1)$, there exists a point $ r \in (\fr,2\fr)$ such that 
\begin{equation}
\label{ito ineq}
\begin{aligned}
\dP& \left( \int_0^{\tau} \left( \int_{ \partial Q_{R+r} }   u(s,\sigma) d\sigma_{R+r}  \right)^ l ds \geq \fp^ l \right) \\
&\leq \dP\left( \int_0^{\tau}  \left( \int_{\partial Q_R} u(s,\sigma) d\sigma_R \right)^ l ds \geq \fq^ l \right) \\
&\quad\quad+ N  R^{\alpha\left(\frac{d}{2} + \frac{d(d-1)}{\gamma+d} + \frac{L  (d-1)}{\gamma}\right)} e^{\alpha  a R\left( l - \lambda + \frac{L}{\gamma} \right)} \fr^{-\alpha\left( 1+ l+\frac{d(\gamma+d-1)}{\gamma+d} \right)} \left(\frac{\fq^L }{\fp}\right)^{\alpha l},
\end{aligned}
\end{equation}
where $N=N(\alpha,\gamma,\lambda,a,d,f,H,K,T)>0$ and $L := \frac{\gamma+1}{\gamma l+1} = \frac{\gamma(\gamma+d)+\gamma+d}{\gamma(\gamma\lambda+d)+\gamma+d}>1$.
\end{enumerate}

\end{lemma}
\begin{remark}

Using Lemma \ref{ito_prop} \eqref{key_prop_1}, we have that for all $R>R_0 \vee 1 $ and $\theta>\theta'>0$, 
\begin{equation}\label{eq:monotonicity}
	 \dP \left(  \int_0^\tau \left( \int_{\partial Q_{R+\theta}} u(s,\sigma) d\sigma_{R+\theta} \right)^ l ds > 0 \right) \leq  \dP \left(  \int_0^\tau \left( \int_{\partial Q_{R+\theta'}}  u(s,\sigma) d\sigma_{R+\theta'} \right)^ l ds > 0 \right).
\end{equation} Indeed, Lemma \ref{ito_prop} \eqref{key_prop_1} implies that 
\begin{equation*}
\begin{aligned}
	\dP \left(  \int_0^\tau \left( \int_{\partial Q_{R+\theta'}}  u(s,\sigma) d\sigma_{R+\theta'} \right)^ l ds = 0 \right) &\leq \dP \Big( u(s,x) =0 \quad \text{for all }s\in[0,\tau] \text{ and } x\in Q_{R+\theta'}\Big)\\
	&\leq \dP \left(  \int_0^\tau \left( \int_{\partial Q_{R+\theta}} u(s,\sigma) d\sigma_{R+\theta} \right)^ l ds = 0 \right),
\end{aligned}
	\end{equation*} which proves \eqref{eq:monotonicity}.

\end{remark}
\vspace{2mm}

We now provide  the proof of Theorem \ref{thm:cpt_support}. 

\begin{proof}[\bf Proof of Theorem \ref{thm:cpt_support}]
Let $\tau\leq T$ be the given bounded stopping time and $u$ be the nonnegative solution introduced in Definition \ref{definition_of_solution}. For $R > 0$, we set 
\begin{equation*}
\Omega_R := \{ \omega:u(\omega,t,x) = 0\quad \mbox{for all}\quad t\leq \tau,\quad x\in Q_R \}.
\end{equation*}
We show that $\dP(\cup_R \Omega_R) = 1$. Since $\Omega_{R_1}\subset \Omega_{R_2}$ for $R_1 \leq R_2$, it suffices to show
\begin{equation*}
\lim_{R\to\infty}\dP(\Omega_R) = 1.
\end{equation*}
By Lemma \ref{ito_prop} \eqref{key_prop_1}, we have
\begin{equation*}
\dP\left(   \int_0^{\tau}\left( \int_{\partial Q_R}  u(s,\sigma)  d\sigma_R \right)^ l ds = 0 \right) \leq \dP(\Omega_R).
\end{equation*}
Therefore, it suffices to show that
\begin{equation}
\limsup_{R\to\infty} \dP\left( \int_0^{\tau} \left( \int_{\partial Q_R}  u(s,\sigma)  d\sigma_R \right)^ l ds >0 \right) = 0,
\end{equation}
Now, let $\gamma \in (0,\eta/2)$ be a constant where $\eta$ is the constant in Assumption~\ref{assumption on f} and $a > 0$ be fixed. For $n\in\dN$, define a bounded stopping time $\tau_n$ such that
\begin{equation*}
\tau_n := n \,\wedge\, \tau \wedge \, \inf\left\{ t\geq0:\| \Psi_a u \|_{C^\gamma([0,t]\times\dR^d)} \geq n \right\},
\end{equation*}
where $\Psi_a$ is the function introduced in \eqref{def of Psi}. Note that stopping $\tau_n$ is well-defined due to Theorem \ref{thm:existence of solution}. Additionally, notice that $\tau_n\to\tau$ almost surely due to \eqref{holder regularity of the solution}.
Observe that
\begin{equation*}
	\begin{aligned}
		&\dP\left( \int_0^{\tau} \left( \int_{\partial Q_R}  u(s,\sigma)  d\sigma_R \right)^ l ds >0 \right) \\
		& \quad \leq \dP\left( \int_0^{\tau_n} \left( \int_{\partial Q_R}  u(s,\sigma)  d\sigma_R \right)^ l ds >0 \right) + 
		\dP\left( \tau_n < \tau\right).
	\end{aligned}
\end{equation*}
Since $\tau_n\to\tau$ almost surely, it is enough to prove that
\begin{equation}\label{eq: claim in lemma}
\limsup_{R\to\infty} \dP\left( \int_0^{\tau_n} \left( \int_{\partial Q_R}  u(s,\sigma)  d\sigma_R \right)^ l ds >0 \right) = 0.
\end{equation}
For proof of \eqref{eq: claim in lemma}, let us denote $\tau$ instead of $\tau_n$ for simplicity. In addition, we may assume $H>0$  satisfies
\begin{equation}\label{eq:Boundedness of time and the solution}
	\| \Psi_a u\|_{C^\gamma([0,\tau]\times\dR^d)} \leq H.
\end{equation}

Let $\xi\in (0,1)$  be a constant that will be specified later.  Choose $\ep>0$ and set $r_k := \ep (k+1)^{-2}$, $p_k := \xi e^{-k}$ for $k = 0,1,2,\dots$ such that $r_0$ satisfies the condition on $\fr$ in \eqref{condition of r for decay prop}. Let $\alpha \in (0,1 )$. By Lemma \ref{ito_prop} \eqref{key_prop_2} for all $R> R_0 \vee 1 $, there exists $\theta_1\in (r_0,2r_0)$ such that 
\begin{equation*}
\begin{aligned}
&\dP\left( \int_0^{\tau} \left( \int_{\partial Q_{R+\theta_1}}  u(s,\sigma)  d\sigma_{R+\theta_1} \right)^l ds \geq p_1^ l \right) \\
&\quad\leq \dP\left(  \int_0^{\tau}\left( \int_{\partial Q_R}  u(s,\sigma)  d\sigma_R \right)^ l ds \geq p_{0}^ l \right) \\
&\quad\quad+ NR^{\alpha\left(\frac{d}{2} + \frac{d(d-1)}{\gamma+d} + \frac{L  (d-1)}{\gamma}\right)} e^{\alpha  a R\left( l - \lambda + \frac{L}{\gamma} \right)} r_0^{-\alpha\left( 1+ l+\frac{d(\gamma+d-1)}{\gamma+d} \right)}\left(\frac{p_{0}^{L}}{p_{1}}\right)^{\alpha l},
\end{aligned}
\end{equation*}
where $N = N(\alpha,\gamma,\lambda,a, d,f,H, K, T)$ and $L := \frac{\gamma+1}{\gamma l+1} = \frac{\gamma(\gamma+d)+\gamma+d}{\gamma(\gamma\lambda+d)+\gamma+d}>1$. Now, we proceed the iteration as follows: For $k=2,3,\dots$, we take $R+\theta_{k-1}$ instead of $R$ and again apply Lemma \ref{ito_prop} \eqref{key_prop_2} to find $\theta_k$ such that $\theta_{k} - \theta_{k-1}\in (r_{k-1},2r_{k-1})$ and
\begin{equation}
    \label{iteration}
\begin{aligned}
&\dP\left( \int_0^{\tau} \left( \int_{\partial Q_{R+\theta_k}} u(s,\sigma) d\sigma_{R+\theta_k} \right)^ l ds \geq p_k^ l \right) \\
&\quad\leq \dP\left(  \int_0^{\tau} \left( \int_{\partial Q_{R+\theta_{k-1}}} u(s,\sigma) d\sigma_{R+\theta_{k-1}} \right)^ l ds \geq p_{k-1}^ l \right) \\
&\quad\quad+ NR^{\alpha\left(\frac{d}{2} + \frac{d(d-1)}{\gamma+d} + \frac{L  (d-1)}{\gamma}\right)} e^{\alpha  a R\left( l - \lambda + \frac{L}{\gamma} \right)} r_{k-1}^{-\alpha\left( 1+ l+\frac{d(\gamma+d-1)}{\gamma+d} \right)}\left(\frac{p_{k-1}^{L}}{p_k}\right)^{\alpha l},
\end{aligned}
\end{equation}
where $N = N(\alpha,\gamma,\lambda,a,d,f,H, K, T)$. 
Here, we used the fact that 
\begin{equation}\label{eq:upper bound of theta_k}
	0 < R+ \theta_{k} \leq R + 2\sum_{j=0}^k r_j \leq R + \frac{\pi^2\ep}{3} < 5R
\end{equation}
for all $k\geq 1$. Due to $|\theta_k-\theta_{k-1}|\leq 2r_{k-1}=2\ep k^{-2}$ for all $k=1,2,\dots$, there exists $\theta$ such that $\theta_k \rightarrow \theta$ as $k\rightarrow \infty$. From \eqref{iteration}, by summing over $k=1,2,\dots$, we have  
\begin{equation}\label{eq: sum of iteration}
\begin{aligned}
&\dP\left( \int_0^{\tau} \left( \int_{\partial Q_{R+\theta}} u(s,\sigma)  d\sigma_{R+\theta} \right)^ l ds > 0 \right) \\
&\quad \leq \dP\left(  \int_0^{\tau}\left( \int_{\partial Q_{R}} u(s,\sigma) d\sigma_R \right)^ l ds \geq \xi^ l \right) + N_1 R^{\alpha\left(\frac{d}{2} + \frac{d(d-1)}{\gamma+d} + \frac{L  (d-1)}{\gamma}\right)} e^{\alpha  a R\left( l - \lambda + \frac{L}{\gamma} \right)} \xi^{\alpha l(L-1)},
\end{aligned}
\end{equation}
where
$$N_1 := N \sum_{k=1}^\infty  r_{k-1}^{-\alpha\left( 1+ l+\frac{d(\gamma+d-1)}{\gamma+d} \right)}e^{\alpha l(k-(k-1)L)}<\infty.
$$ 
The last inequality holds thanks to $L>1$.
Choose $\delta>\frac{a\left( l-\lambda+\frac{L}{\gamma} \right)}{l(L-1)}$ and set $\xi := e^{-\delta R}$, where $K$ is the constant introduced in Assumptions \ref{assumptions on coefficients} and \ref{assumptions on h}. Then, by Chebyshev's and Jensen's inequalities,
\begin{equation*}
\begin{aligned}
\dP\left( \int_0^\tau\left( \int_{\partial Q_R} u(s,\sigma) d\sigma_R \right)^ l ds \geq \xi^ l \right) &\leq \xi^{- l}\dE \left[\int_0^\tau\left( \int_{\partial Q_R} u(s,\sigma) d\sigma_R \right)^ l ds  \right]\\
&\leq \xi^{- l}N\left(\dE \left[\int_0^\tau \int_{\partial Q_R} u(s,\sigma) d\sigma_R  ds \right]\right)^ l \\
&\leq N\left(\dE \left[\int_0^\tau e^{\delta R}\int_{\partial Q_R}  u(s,\sigma) d\sigma_R ds \right]\right)^ l,
\end{aligned}
\end{equation*}
where $N = N(\gamma,\lambda,d,T)$. Thus, by Lemma \ref{L_1 bound with weight}  (see \eqref{eq:integral bound with weight}), 
\begin{equation}
\label{to control limit1}
\limsup_{R\to\infty}\dP\left( \int_0^\tau\left( \int_{\partial Q_R} u(s,\sigma) d\sigma_R \right)^ l ds \geq \xi^ l \right) = 0.
\end{equation}
In addition, since  $L >1$, we have 
\begin{equation}
\label{to control limit2}
\begin{aligned}
&\limsup_{R\to\infty}R^{\alpha\left(\frac{d}{2} + \frac{d(d-1)}{\gamma+d} + \frac{L  (d-1)}{\gamma}\right)} e^{\alpha  a R\left( l - \lambda + \frac{L}{\gamma} \right)}\xi^{(L-1)\alpha l} \\
&\quad= \limsup_{R\to\infty}R^{\alpha\left(\frac{d}{2} + \frac{d(d-1)}{\gamma+d} + \frac{L  (d-1)}{\gamma}  \right)}  e^{\alpha  a R\left( l - \lambda + \frac{L}{\gamma} \right) - \delta \alpha l (L-1)R} \\
&\quad  = 0.
\end{aligned}
\end{equation}
Therefore, by applying \eqref{to control limit1} and \eqref{to control limit2} to \eqref{eq: sum of iteration}, we have
\begin{equation*}
\begin{aligned}
&\limsup_{R\to\infty}\dP\left( \int_0^\tau\left( \int_{\partial Q_R} u(s,\sigma) d\sigma_R \right)^ l ds > 0 \right) \\
&=\limsup_{R\to\infty}\dP\left( \int_0^\tau\left( \int_{\partial Q_{R+\theta}} u(s,\sigma) d\sigma_R \right)^ l ds > 0 \right) \\
&\leq \limsup_{R\to\infty}\left\{\dP\left( \int_0^\tau\left( \int_{\partial Q_R} u(s,\sigma) d\sigma_R \right)^ l ds \geq \xi^ l \right) + N_1 R^{\alpha\left(\frac{d}{2} + \frac{d(d-1)}{\gamma+d} + \frac{L  (d-1)}{\gamma}\right)} e^{\alpha  a R\left( l - \lambda + \frac{L}{\gamma} \right)} \xi^{\alpha l(L-1)}\right\} \\
&= 0,
\end{aligned}
\end{equation*} 
which completes the proof.

\end{proof}


\section{Proof of Lemma \ref{ito_prop}}\label{subsec:proof of props}

In this section, we provide the proof of Lemma \ref{ito_prop}. To prove Lemma \ref{ito_prop}, we assume the results of Theorem \ref{thm:existence of solution} and Lemma \ref{L_1 bound with weight}, and introduce three auxiliary lemmas; Lemmas \ref{prop:estimation of quadratic variation}, \ref{lemma:integral inequality}, and \ref{lemma:lower bound of covariance}.  Lemma \ref{prop:estimation of quadratic variation} shows that the quadratic variation of the stochastic part in \eqref{eq:sol_int_eq_form} on $Q_R$ is controlled by the integration of the solution $u$ on $\partial Q_R$.

Throughout this section, we assume that there exist $a,H\in(0,\infty)$ satisfying \eqref{holder reg for ito ineq}. For $R>1$, we define 
\begin{equation}
\label{def of M(x)}
M(x) := M_R(x) := (|x| - R)_+
\end{equation}
and choose $K_1$ and $K_2$ such that
\begin{equation}
\label{ineq:condition_of_K1_and_K2}
\begin{aligned}
K_1  >  2K^2\vee a   \quad \text{and} \quad K_2 \geq  K + 4KK_1 + K_1^2K =: N_1.
\end{aligned}
\end{equation} where $K$ is the constant in Assumption~\ref{assumptions on coefficients}.
\begin{lemma}
\label{prop:estimation of quadratic variation}
Let $\tau\leq T$ be a bounded stopping time and $u$ be a nonnegative solution to \eqref{original_equation} introduced in Definition \ref{definition_of_solution}. Suppose $R> R_0 \vee 1$ is a constant, where $R_0$ is the constant introduced in Assumption \ref{assumption on initial data}. Then, for every  $\alpha\in(0,1)$, we have
\begin{equation}
\label{eq:estimation of quadratic variation}
\begin{aligned}
&\dE\left[\left( \int_0^\tau\int_{\dR^d}\int_{\dR^d} e^{-2K_2s-K_1(M(x+y)+M(y))}M(x+y)M(y)(u(s,x+y)u(s,y))^\lambda dy f(dx) ds \right)^{\alpha/2} \right]\\
&\quad \leq N \dE\left[\left(\int_0^\tau \int_{ \partial Q_R } e^{-K_2s}u(s,\sigma) d\sigma_R ds\right)^\alpha\right],\\
\end{aligned}
\end{equation}
where $N := N(\alpha,d,K)$.
\end{lemma}

\begin{proof}

To consider a weak solution $u$ of \eqref{original_equation} on $Q_R$, we need to choose appropriate test functions. Let $\psi\in C_c^\infty(\dR)$ be a nonnegative symmetric function satisfying $\psi(z) = 1$ on $|z|<1$ and $\psi(z) = 0$ on $|z|\geq2$ and set $\psi_n(x) := \psi_n(|x|) := \psi(|x|/n)$ for $n\in \dN$. Take a nonnegative function $\zeta\in C_c^\infty(\dR)$ satisfying $\int_{\dR}\zeta(z)dz = 1$, symmetric, and $\zeta = 0$ on $|z|\geq1$. Define for $m \in \dN$,
\begin{equation}\label{eq:defintion of phi_m}
	\phi_m (x) := m\int_{\dR} \bi_{|z|>R + \frac{1}{m}}(z) \zeta \left(m \left(|x| - z\right)\right) dz = \int_{\dR} \bi_{\left||x|-\frac{1}{m}z\right|>R + \frac{1}{m}}(z) \zeta \left(z\right) dz.
\end{equation} 
It should be noted that $\phi_m(x)=  0$ on $|x|\leq R$ and $\phi_m(x) \to \bi_{|x|>R}(x)$ as $m\to\infty$. Additionally, observe that
\begin{equation}
\label{phi m = 1}
\phi_m(x) = 1
\end{equation}
on $|x|>R+\frac{2}{m}$. Indeed, for $|x| > R+\frac{2}{m}$ and $|z|<1$, we have $\left| |x| - \frac{1}{m}z \right|>|x| - \frac{|z|}{m}>R+\frac{2}{m} - \frac{|z|}{m}>R+\frac{1}{m}$. Therefore, by the definition of $\phi_m$, \eqref{eq:defintion of phi_m} implies $\eqref{phi m = 1}.$ Moreover, \eqref{phi m = 1} ensures that 
\begin{equation}
\label{derivatives of phi m}
\phi_{mx^i}(x) = \phi_{mx^ix^j}(x) = 0 \quad \text{for $|x|>R+\frac{2}{m}$. }
\end{equation}
Furthermore, if $\left| |x| - \frac{z}{m} \right|>R+\frac{1}{m}$ and $|z|<1$, then we have $R  < \left| |x|-\frac{z}{m} \right| - \frac{1}{m} \leq |x| + \frac{|z|}{m} - \frac{1}{m} \leq  |x|.$ Thus,
\begin{equation}
\label{domain control}
\phi_m(x)\leq \bi_{|x|>R}.
\end{equation}
For $g\in L_{1,loc}(\dR^d)$ and $\ep>0$, set 
\begin{equation}
	\label{def of mollification}
	g^{(\ep)}(x) := \ep^{-d}\int_{\dR^d} g(y)\zeta(\ep^{-1}(x-y)) dy
\end{equation}
Then, for $m,n\in\dN$, $\ep>0$, and $\ep_1>0$, It\^o's formula yields
\begin{equation}
\label{ineq:applying ito's formula}
\begin{aligned}
& \int_{\dR^d} e^{-K_2\tau}u(\tau,x)M^{(\ep)}(x)\psi_n(x)\phi_m(x)e^{-K_1M^{(\ep_1)}(x)} dx \\
& = \int_0^\tau\int_{\dR^d} e^{-K_2s}u(s,x) \cL^* \left(  M^{(\ep)}(x)\psi_n(x) \phi_m(x)e^{-K_1M^{(\ep_1)}(x)} \right)  dxds \\
& \quad -K_2 \int_0^\tau\int_{\dR^d} e^{-K_2s}u(s,x)M^{(\ep)}(x)\psi_n(x) \phi_m(x)e^{-K_1M^{(\ep_1)}(x)}dxds + \fm^{m,n,\ep,\ep_1}_\tau,
\end{aligned}
\end{equation}
where
$$ \fm^{m,n,\ep,\ep_1}_\tau := \int_0^\tau\int_{\dR^d} 
 e^{-K_2s}h(s,x,u(s,x))M^{(\ep)}(x)\psi_n(x) \phi_m(x)e^{-K_1M^{(\ep_1)}(x)} F(ds,dx).
$$ 
For the properties of $\fm^{m,n,\ep,\ep_1}_\tau$, see Remark \ref{random measure} and \eqref{def of covaiance functional}. Note that the quadratic variation of $\fm^{m,n,\ep,\ep_1}_t$ is related to the LHS of \eqref{eq:estimation of quadratic variation}. We will get the inequality \eqref{eq:estimation of quadratic variation} by taking various limits in \eqref{ineq:applying ito's formula} and using some martingale theory. The remaining part of the proof is separated into four steps. From \textbf{Step 1} to \textbf{Step 3}, the limits are taken in $\ep_1\to0$, $\ep\to0$, and then  $m\to\infty$. In \textbf{Step 4}, we take the limit $n\to\infty$ \eqref{ineq:applying ito's formula} and obtain \eqref{eq:estimation of quadratic variation}.

\vspace{1mm}

\textbf{(Step 1)} We consider the case $\ep_1\downarrow0$. For convenience, set 
\begin{equation}
\label{notation Phi for convenience}
\Phi^{m,n,\ep}(x) : =  M^{(\ep)}(x)\psi_n(x)\phi_m(x).
\end{equation}
Then, from \eqref{ineq:applying ito's formula} and \eqref{eq:adjoint}, we have
\begin{equation}
\label{applying adjoint operator to test function}
\begin{aligned}
& \cL^* \left( \Phi^{m,n,\ep}(x)e^{-K_1M^{(\ep_1)}(x)} \right) \\
& = \left( a^{ij}_{x^ix^j} -  b^i_{x^i} + c \right)\Phi^{m,n,\ep}(x)e^{-K_1M^{(\ep_1)}(x)} + \left( 2 a^{ij}_{x^j} - b^i \right)\left( \Phi^{m,n,\ep}(x)e^{-K_1M^{(\ep_1)}(x)} \right)_{x^i}\\
&\quad\quad +  a^{ij} \left( \Phi^{m,n,\ep}(x)e^{-K_1M^{(\ep_1)}(x)} \right)_{x^ix^j} \\
& = \left[ 2 a^{ij}_{x^j} - b^i -2 K_1  a^{ij}M^{(\ep_1)}_{x^j}(x)  \right]\left( \Phi^{m,n,\ep}(x) \right)_{x^i}  e^{-K_1M^{(\ep_1)}(x)} \\
&\quad\quad  + a^{ij} \left(\Phi^{m,n,\ep}(x) \right)_{x^ix^j} e^{-K_1M^{(\ep_1)}(x)} + {\bf D}^{\ep_1}_1(s,x)  \Phi^{m,n,\ep}(x)e^{-K_1M^{(\ep_1)}(x)},
\end{aligned}
\end{equation} 
where 
\begin{equation}
\label{def of D1ep}
\begin{aligned}
{\bf D}^{\ep_1}_1(s,x) &= a^{ij}_{x^ix^j} -  b^i_{x^i} + c -K_1  \left( 2 a^{ij}_{x^j} - b^i \right) M^{(\ep_1)}_{x^i}(x) - K_1  a^{ij} M^{(\ep_1)}_{x^ix^j}(x) + K_1^2  a^{ij} M^{(\ep_1)}_{x^i}(x)M^{(\ep_1)}_{x^j}(x).
\end{aligned}
\end{equation}
To deal with ${\bf D}^{\ep_1}_1(s,x)$, notice that if we let $\ep_1\downarrow0$, then on $|x|>R$,
\begin{equation}\label{eq:convergence of M(x)}
M^{(\ep_1)}(x)\rightarrow M(x), \quad M^{(\ep_1)}_{x^i}(x) \rightarrow M_{x^i}(x), \quad  M^{(\ep_1)}_{x^ix^j}(x)\rightarrow M_{x^ix^j}(x),
\end{equation} 
where 
\begin{equation}\label{eq:derivative of M(x)}
	 M_{x^i}(x) = \frac{x^i}{|x|}, \quad M_{x^ix^j}(x) = \begin{cases}
	 \frac{\delta^{ij}}{|x|} - \frac{x^ix^j}{|x|^3} \quad&\text{if}\quad d \geq 2, \\
	 \hfil 0 \quad&\text{if}\quad d = 1,
	 \end{cases}
\end{equation} 
and $\delta^{ij}$ is the Kronecker delta.
Thus, by applying \eqref{eq:convergence of M(x)} and \eqref{eq:derivative of M(x)} to \eqref{def of D1ep}, on $|x|>R$, we have 
\begin{equation}
\label{limit of D1ep}
\begin{aligned}
&\limsup_{\ep_1\to0}{\bf D}^{\ep_1}_1(s,x)\\
&\quad=a^{ij}_{x^ix^j}  -  b^i_{x^i} + c -K_1 \left( 2 a^{ij}_{x^j} - b^i \right) M_{x^i} - K_1  a^{ij} M_{x^ix^j} + K_1^2  a^{ij} M_{x^i}M_{x^j} \leq N_1,
\end{aligned}
\end{equation}
where $N_1$ is the constant introduced in \eqref{ineq:condition_of_K1_and_K2}.
Therefore, if we take the limit in probability as $\ep_1\rightarrow 0$ in \eqref{ineq:applying ito's formula}, \eqref{limit of D1ep} yields
\begin{equation}\label{nonnegativity of A+B+m}
\begin{aligned}
&\int_{\dR^d} e^{-K_2\tau}u(\tau,x)M^{(\ep)}(x)\psi_n(x)\phi_m(x)e^{-K_1M(x)} dx \\
&\quad \leq \int_0^\tau e^{-K_2 s} \left({\bf{A}}(m,n,\ep) + {\bf{B}}(m,n,\ep) \right) \, ds + \fm_\tau^{m,n,\ep} \,
	\end{aligned}
\end{equation}
where 
\begin{align}
{\bf{A}}(m,n,\ep) &:=  \int_{\dR^d}\left[ 2 a^{ij}_{x^j} - b^i - 2K_1 a^{ij}M_{x^j}(x) \right]  u(s,x)\left(\Phi^{m,n,\ep}(x)\right)_{x^i}e^{-K_1M(x)} dx, \label{def of A} \\
{\bf{B}}(m,n,\ep)&:= \int_{\dR^d} u(s,x)  \left(a^{ij}\left(\Phi^{m,n,\ep}(x)\right)_{x^ix^j} - ( K_2 - N_1 )\Phi^{m,n,\ep}(x) \right)e^{-K_1M(x)} dx,  \label{def of B}\\
\fm_\tau^{m,n,\ep} &:= \int_0^\tau\int_{\dR^d} e^{-K_2s}h(s,x,u(s,x))M^{(\ep)}(x)\psi_n(x) \phi_m(x)e^{-K_1M(x)} F(ds,dx). \label{def of m}
\end{align}
Note that one can show that  $\fm_\tau^{m,n,\ep,\ep_1}$ converges in probability to $\fm_\tau^{m,n,\ep}$ as $\ep_1\downarrow 0$  by the Burkholder-Davis-Gundy inequality, the bounded convergence theorem, and the fact that $\psi_n \in C^\infty_c(\dR^d)$.

\vspace{1mm}

\textbf{(Step 2)} In this step, our aim is to show that 
\begin{equation}\label{eq:estimation of A}
\begin{aligned}
		\limsup_{m\rightarrow \infty}\limsup_{\ep \rightarrow 0} {\bf A}(m,n,\ep) \leq - &K^{-1}K_1  \int_{Q_R}u(s,x) \psi_n(x) e^{-K_1M(x)} dx\\
		&+N \sum_i\int_{Q_R}  u(s,x)M(x)\psi_{nx^i}(x) e^{-K_1M(x)} dx,
\end{aligned}
\end{equation}
where ${\bf A}(m,n,\ep)$ is  in \eqref{def of A}  and $N = N(K)$. We separate ${\bf A}(m,n,\ep)$ into three parts:
\begin{equation*}
	\begin{aligned}
		{\bf A}(m,n,\ep) = {\bf A}_1(m,n,\ep)+{\bf A}_2(m,n,\ep)+{\bf A}_3(m,n,\ep),
	\end{aligned}
\end{equation*}
where
\begin{align*}
&{\bf A}_1(m,n,\ep) := \int_{\dR^d} \left( 2 a^{ij}_{x^j} - b^i - 2K_1 a^{ij}M_{x^j}(x)\right) u(s,x)M_{x^i}^{(\ep)}(x)\psi_n(x) \phi_m(x)e^{-K_1M(x)} dx, \\
&{\bf A}_2(m,n,\ep) :=  \int_{\dR^d} \left( 2 a^{ij}_{x^j} - b^i - 2K_1 a^{ij}M_{x^j}(x)\right) u(s,x)M^{(\ep)}(x)\psi_{nx^i}(x) \phi_m(x)e^{-K_1M(x)} dx, \\
&{\bf A}_3(m,n,\ep) :=  \int_{\dR^d} \left( 2 a^{ij}_{x^j} - b^i - 2K_1 a^{ij}M_{x^j}(x)\right) u(s,x)M^{(\ep)}(x)\psi_n(x) \phi_{mx^i}(x)e^{-K_1M(x)} dx. \\
\end{align*}
We use  \eqref{eq:derivative of M(x)}, Assumption \ref{assumptions on coefficients}, and \eqref{ineq:condition_of_K1_and_K2} to get that 
\begin{equation}\label{eq:estimation of A_1}
\begin{aligned}
&\limsup_{m\to\infty}\limsup_{\ep\rightarrow 0}{\bf A}_1(m,n,\ep) \\
&\quad = \limsup_{m\to\infty}\int_{\dR^d}u(s,x) \left[ 2 a^{ij}_{x^j}\frac{x^i}{|x|} -  b^i\frac{x^i}{|x|} - 2K_1 a^{ij}\frac{x^ix^j}{|x|^2} \right]\psi_n(x) \phi_{m}(x) e^{-K_1M(x)}  d x \\
&\quad\leq \left[ 2K - 2K^{-1} K_1 \right] \limsup_{m\to\infty}\int_{\dR^d}u(s,x) \psi_n(x) \phi_m(x)e^{-K_1M(x)} dx \\
&\quad\leq - K^{-1}K_1  \int_{Q_R}u(s,x) \psi_n(x) e^{-K_1M(x)} dx.
\end{aligned}
\end{equation} 
In addition, since $|M_{x^i}| \leq 1 $ (see \eqref{eq:derivative of M(x)}), we have
\begin{equation}\label{eq:estimation of A_2}
	\begin{aligned}
		\limsup_{m\to\infty}&\limsup_{\ep\rightarrow 0}{\bf A}_2(m,n,\ep)\\
		& = \limsup_{m\to\infty}\int_{\dR^d}\left[ 2 a^{ij}_{x^j} - b^i - 2K_1 a^{ij}M_{x^j}(x) \right]  u(s,x)M(x)\psi_{nx^i}(x) \phi_m(x)e^{-K_1M(x)} dx\\
		& \leq [2K+2KK_1]\limsup_{m\to\infty} \int_{\dR^d}  u(s,x)M(x) \left| \psi_{nx^i}(x) \right| \phi_m(x)e^{-K_1M(x)} dx \\
		& \leq N \int_{Q_R}  u(s,x)M(x)  \left| \psi_{nx^i}(x) \right| e^{-K_1M(x)} dx,
	\end{aligned}
\end{equation} 
where $N = N(K)$.
To control ${\bf A}_3(m,n,\ep)$, observe that \eqref{derivatives of phi m} and \eqref{domain control} imply
\begin{equation}
\label{eq:derivation of estimation of A_3} 
	\begin{aligned}
&\limsup_{\ep \rightarrow 0} \left| {\bf A}_3(m,n,\ep) \right| \\
&\quad \leq [2K+2KK_1]m\int_{|x|<R+\frac{2}{m}} u(s,x) M(x)\psi_n(x) \left|\frac{x^i}{|x|}\int_{\dR}\bi_{||x|-\frac{1}{m}z|>R+\frac{1}{m}}\zeta'(z)dz\right| e^{-K_1 M(x)}dx \\
&\quad\leq Nm \int_{|x|<R+\frac{2}{m}} u(s,x)M(x) \psi_n(x)\bi_{|x|>R}e^{-K_1M(x)}dx \\
&\quad= Nm^{-1}\int_0^2\int_{\partial Q_{R+m^{-1}\rho}} u(s,\sigma)\, \rho \, \psi_n(R+m^{-1}\rho)  e^{-m^{-1}K_1\rho} d\sigma_{R+m^{-1}\rho} d\rho,
	\end{aligned}
\end{equation} 
where $N = N(d,K)$. Since all terms in the last integral are bounded and continuous almost surely, we have 
\begin{equation}\label{eq:estimation of A_3}
	\limsup_{m\rightarrow \infty} \limsup_{\ep \rightarrow 0}\left| {\bf A}_3(m,n,\ep) \right| =0.
\end{equation} 
By combining \eqref{eq:estimation of A_1}, \eqref{eq:estimation of A_2}, and \eqref{eq:estimation of A_3}, we have \eqref{eq:estimation of A}. 

\vspace{1mm}

\textbf{(Step 3)} Now we control ${\bf B}(m,n,\ep)$. In this step, we show that for fixed $n > R$,
\begin{equation}\label{eq:estimation of B}
\begin{aligned}
&\limsup_{m\rightarrow \infty}\limsup_{\ep \rightarrow 0} {\bf B}(m,n,\ep) \\
&\quad\leq N \int_{Q_R} u(s,x) \left(|M_{x^i}(x)\psi_{nx^j}(x)| + M(x)|\psi_{nx^ix^j}(x)| \right) e^{-K_1M(x)}dx \\
& \quad\quad + 2K\int_{Q_R} u(s,x) \psi_n(x) e^{-K_1M(x)}dx + K\int_{\partial Q_R} u(s,\sigma)d\sigma_R,
\end{aligned}
\end{equation} 
where $N:=N(d,K)$ and ${\bf B}(m,n,\ep)$ are introduced in \eqref{def of B}.
We decompose ${\bf B}(m,n,\ep)$ as 
\begin{equation}
	\begin{aligned}
		{\bf B}(m,n,\ep) = {\bf B}_1(m,n,\ep)+{\bf B}_2(m,n,\ep)+{\bf B}_3(m,n,\ep)+{\bf B}_4(m,n,\ep)+{\bf B}_5(m,n,\ep)+{\bf B}_6(m,n,\ep),
	\end{aligned}
\end{equation} 
where
\begin{equation*}
\begin{aligned}
&{\bf B}_1(m,n,\ep) := \left( N_1 - K_2 \right)\int_{\dR^d} u(s,x)  M^{(\ep)}(x) \psi_n(x)\phi_m(x) e^{-K_1M(x)}dx, \\
&{\bf B}_2(m,n,\ep) :=  \int_{\dR^d} u(s,x) a^{ij}(s,x) M^{(\ep)}_{x^ix^j}(x) \psi_n(x) \phi_m(x)e^{-K_1M(x)}dx, \\
&{\bf B}_3(m,n,\ep) := 2\int_{\dR^d} u(s,x)  a^{ij}(s,x) M^{(\ep)}_{x^i}(x)\psi_n(x) \phi_{mx^j}(x)e^{-K_1M(x)} dx, \\
&{\bf B}_4(m,n,\ep) := 2\int_{\dR^d} u(s,x)  a^{ij}(s,x) M^{(\ep)}(x) \psi_{nx^i}(x) \phi_{mx^j}(x)e^{-K_1M(x)} dx, \\
&{\bf B}_5(m,n,\ep) := \int_{\dR^d} u(s,x) a^{ij}(s,x) M^{(\ep)}(x)\psi_n(x) \phi_{mx^ix^j}(x) e^{-K_1M(x)}dx, \\
&{\bf B}_6(m,n,\ep) := \int_{\dR^d} u(s,x) a^{ij}(s,x) \left(2M^{(\ep)}_{x^i}(x)\psi_{nx^j}(x) + M^{(\ep)}(x)\psi_{nx^ix^j}(x) \right)\phi_m(x) e^{-K_1M(x)}dx.
\end{aligned}
\end{equation*}
In the case of ${\bf B}_1(m,n,\ep)$, by the choice of $K_2$ (see \eqref{ineq:condition_of_K1_and_K2}), we have 
\begin{equation}\label{eq:estimate of B_1}
	\begin{aligned}
		\limsup_{m\to\infty}\limsup_{\ep \rightarrow 0}{\bf B}_1(m,n,\ep) 
		\leq 0.
	\end{aligned}
\end{equation}
In the case of ${\bf B}_2(m,n,\ep)$, \eqref{eq:convergence of M(x)} yield
\begin{equation}\label{eq:estimate of B_2}
\begin{aligned}
&\limsup_{m\to\infty}\limsup_{\ep \rightarrow 0} \left|{\bf B}_2(m,n,\ep)\right| \\
&\quad = \limsup_{m\to\infty}\int_{\dR^d}u(s,x)\left| a^{ij}(s,x)\left( \frac{\delta^{ij}}{|x|} - \frac{x^ix^j}{|x|^3}  \right)\right|\psi_n(x)\phi_m(x)e^{-K_1M(x)}dx\\
&\quad \leq  2K \int_{Q_R} u(s,x) \psi_n(x) e^{-K_1M(x)}dx.
\end{aligned}
\end{equation}
To bound ${\bf B}_3(m,n,\ep)$, we mimic the estimation for ${\bf A}_3(m,n,\ep)$; see \eqref{eq:derivation of estimation of A_3}. Then, we have 
\begin{equation*}
\begin{aligned}
& \lim_{\ep \rightarrow 0} {\bf B}_3(m,n,\ep) \\
&\quad =  2 m \int_{|x|<R+\frac{2}{m}} u(s,x) a^{ij}(s,x) \frac{x^ix^j}{|x|^2} \psi_n(x) \int_{\dR} \bi_{\left| |x|-\frac{1}{m}z \right|>R+\frac{1}{m}}\zeta'(z) dz e^{-K_1M(x)}dx\\
&\quad =  2 m \int_0^{R+\frac{2}{m}}\int_{\partial Q_\rho} u(s,\sigma) a^{ij}(s,\sigma) \frac{\sigma^i\sigma^j}{\rho^2} \psi_n(\rho) \int_{\dR} \bi_{\left| \rho-\frac{1}{m}z \right|>R+\frac{1}{m}}\zeta'(z) dz e^{-K_1(\rho-R)}d\sigma_{R+\frac{\rho}{m}}d\rho\\
&\quad =  2  \int_{0}^{2}\int_{\partial Q_{R+\frac{\rho}{m}}} u(s,\sigma) a^{ij}(s,\sigma) \frac{\sigma^i\sigma^j}{(R+\rho m^{-1})^2}   d\sigma_{R+\frac{\rho}{m}} \psi_n\left( R+\frac{\rho}{m} \right) \zeta(\rho-1) e^{-K_1\rho/m} d\rho.
\end{aligned}
\end{equation*} 
By taking the limit $m\rightarrow \infty$, we have
\begin{equation}\label{eq:estimate of B_3}
\begin{aligned}	
 \lim_{m\rightarrow \infty}\lim_{\ep \rightarrow 0} \, {\bf B}_3(m,n,\ep) &
= 2R^{-2}\psi_n(R)\int_{\partial Q_R} u(s,\sigma) a^{ij}(s,\sigma)\sigma^i\sigma^j d\sigma_R\\
&  = 2R^{-2}\int_{\partial Q_R} u(s,\sigma) a^{ij}(s,\sigma)\sigma^i\sigma^j d\sigma_R,
\end{aligned}	
\end{equation} 
since $\psi_n(R)=1$ for $n > R$.
In the case of ${\bf B}_4(m,n,\ep)$, note that
\begin{equation*}
\begin{aligned}
&\limsup_{\ep \rightarrow 0} \left| {\bf B}_4(m,n,\ep) \right| \\
&\quad \leq  N \int_{|x|<R+\frac{2}{m}} u(s,x) \left| a^{ij}(s,x) \right|  M(x) |\psi_{nx^i}(x)| \left| \phi_{mx^j}(x) \right| e^{-K_1M(x)}dx\\
&\quad\leq Nm \int_{|x|<R+\frac{2}{m}} u(s,x)  M(x) |\psi_{nx^i}(x)| \bi_{|x|>R}(x) e^{-K_1M(x)}dx \\
&\quad\leq Nm^{-1}\int_0^2 \int_{\partial Q_{R+m^{-1}\rho}} u(s,\sigma)  d\sigma_{R+m^{-1}\rho} \left| \psi_{nx^i}(R+m^{-1}\rho) \right|\, \rho e^{-K_1m^{-1}\rho}  d\rho, 
\end{aligned}
\end{equation*} 
where $N:=N(K)$. By taking the limit $m\rightarrow \infty$, we have
\begin{equation}\label{eq:estimate of B_4}
	\begin{aligned}	
		\limsup_{m\rightarrow \infty}\limsup_{\ep \rightarrow 0} &\, \left| {\bf B}_4(m,n,\ep) \right|
		= 0.
			\end{aligned}	
\end{equation} 
To control ${\bf B}_5(m,n,\ep)$, recall that $\phi_{mx^ix^j}(x) =0 $ for  $|x|>R+\frac{2}{m}$ (see \eqref{derivatives of phi m} and \eqref{domain control}).   Thus, similar to the estimation for  ${\bf B}_3(m,  n, \epsilon)$,  we get 
\begin{equation*}
\begin{aligned}
&\lim_{\ep\rightarrow 0}{\bf B}_5(m,n,\ep) \\
& = m^2 \int_{|x|<R+\frac{2}{m}} u(s,x) a^{ij}(s,x)\frac{x^ix^j}{|x|^2} M(x) \psi_n(x) \int_{\dR} \bi_{\left| |x|-\frac{1}{m}z \right|>R+\frac{1}{m}} \zeta''(z) dz e^{-K_1M(x)}dx\\
&\quad + m \int_{|x|<R+\frac{2}{m}} u(s,x) a^{ij}(s,x) \left( \frac{\delta^{ij}}{|x|} - \frac{x^ix^j}{|x|^3} \right) M(x) \psi_n(x) \int_{\dR} \bi_{\left| |x|-\frac{1}{m}z \right|>R+\frac{1}{m}} \zeta'(z) dz e^{-K_1M(x)}dx\\
& = m^2 \int_0^{R+\frac{2}{m}}\int_{\partial Q_\rho} u(s,\sigma) a^{ij}(s,\sigma)\frac{\sigma^i \sigma^j}{\rho^2} d\sigma_{\rho} M(\rho) \psi_n(\rho) \int_{\dR} \bi_{\left| \rho-\frac{1}{m}z \right|>R+\frac{1}{m}} \zeta''(z) dz e^{-K_1M(\rho)}d\rho\\
&\quad + m \int_0^{R+\frac{2}{m}}\int_{\partial Q_\rho} u(s,\sigma) a^{ij}(s,\sigma) \left( \frac{\delta^{ij}}{\rho} - \frac{\sigma^i\sigma^j}{\rho^3} \right) d\sigma_{\rho} M(\rho) \psi_n(\rho) \int_{\dR} \bi_{\left| \rho-\frac{1}{m}z \right|>R+\frac{1}{m}} \zeta'(z) dz e^{-K_1M(\rho)}d\rho\\
& = \int_{0}^2\int_{\partial Q_{R+\frac{1}{m}\rho} } u(s,\sigma) a^{ij}(s,\sigma)\sigma^i \sigma^j d\sigma_{R+m^{-1}\rho} \psi_n\left(R+m^{-1}\rho\right) \zeta'(\rho-1) \frac{\rho}{(R+m^{-1}\rho)^2} e^{-K_1M(R+m^{-1}\rho)} d\rho\\
&\quad + m^{-1} \int_{0}^{2}\int_{\partial Q_{R+\frac{1}{m}\rho}} u(s,\sigma) a^{ij}(s,\sigma) \left( \frac{\delta^{ij}}{R+m^{-1}\rho} - \frac{\sigma^i\sigma^j}{(R+m^{-1}\rho)^3} \right) d\sigma_{\rho} \\
&\quad\quad\quad\times \psi_n\left( R+m^{-1}\rho \right) \zeta(\rho-1)  \rho  e^{-K_1M(R+m^{-1}\rho)} d\rho.\\
\end{aligned}
\end{equation*} 
The bounded convergence theorem implies that 
\begin{equation}\label{eq:estimate of B_5}
\begin{aligned}
\lim_{m\rightarrow \infty } \lim_{\ep\rightarrow 0} {\bf B}_5(m,n,\ep) &= - R^{-2} \psi_n(R) \int_{\partial Q_R}u(s,\sigma) a^{ij}(s,\sigma)\sigma^i\sigma^j d\sigma_R\\
&= - R^{-2}  \int_{\partial Q_R}u(s,\sigma) a^{ij}(s,\sigma)\sigma^i\sigma^j d\sigma_R.
\end{aligned}
\end{equation} 
In the case of ${\bf B}_6(m,n,\ep)$, by \eqref{eq:convergence of M(x)}, we have
\begin{equation}\label{eq:estimate of B_6}
	\begin{aligned}
		&\limsup_{m\rightarrow \infty } \limsup_{\ep\rightarrow 0} \left|{\bf B}_6(m,n,\ep)\right|\\
		&\quad\leq N \int_{Q_R} u(s,x) \left(2|M_{x^i}(x)\psi_{nx^j}(x)| + M(x)|\psi_{nx^ix^j}(x)| \right) e^{-K_1M(x)}dx, 
	\end{aligned}
\end{equation} where $N:=N(K)$. Collecting \eqref{eq:estimate of B_1}, \eqref{eq:estimate of B_2}, \eqref{eq:estimate of B_3}, \eqref{eq:estimate of B_2}, \eqref{eq:estimate of B_4}, \eqref{eq:estimate of B_5} and \eqref{eq:estimate of B_6}, we complete the proof of \eqref{eq:estimation of B}.

\vspace{1mm}

\textbf{(Step 4)} In this step, we consider the limit $n\to\infty$. Before we let $n\to\infty$, notice that $\fm_\tau^{m,n,\ep}$ converges in probability to a local martingale 
\begin{equation*}
\begin{aligned}
\fm^{n}_\tau := \int_0^\tau\int_{Q_R} e^{-K_2s}h(s,x,u(s,x))M(x)\psi_n(x) e^{-K_1M(x)} F(ds,dx)
\end{aligned}
\end{equation*}
as $\ep\downarrow 0$ and $m\to\infty$ due to the same reason as in the end of {\bf  (Step 1)}. 

By the choice of $K_1$, we get 
\begin{equation}
 	(-K^{-1}K_1 + 2K) \int_{Q_R}u(s,x) \psi_n(x) e^{-K_1M(x)} dx \leq 0;
 \end{equation}
 see \eqref{ineq:condition_of_K1_and_K2}. Then, by applying \eqref{eq:estimation of A} and \eqref{eq:estimation of B} to \eqref{nonnegativity of A+B+m}, we have
\begin{equation}\label{eq:nonnegativity of sum}
	\begin{aligned}
		0\leq \int_{Q_R} e^{-K_2\tau}u(\tau,x)M(x)\psi_n(x)e^{-K_1M(x)} dx \leq \fA^n_\tau+\fB_\tau+ \fm^{n}_\tau,
	\end{aligned}
\end{equation} 
where 
\begin{equation}
\label{int frA} 
	\begin{aligned}
		\fA^n_\tau := &Nn^{-1}\int_0^\tau\int_{Q_R} e^{-K_2s}u(s,x)\bigg[ |\psi_{x^i}(x/n)| M(x)\\
		&\quad\quad\quad\quad\quad\quad
		 + \left(|M_{x^i}(x)\psi_{x^j}(x/n)| + M(x)|\psi_{x^ix^j}(x/n)| \right) \bigg] e^{-K_1M(x)}dxds
	\end{aligned}
\end{equation} 
and
\begin{equation*}
	\begin{aligned}
		\fB_t:=  N(d,K) \int_0^\tau e^{-K_2s}\int_{\partial Q_R} u(s,\sigma)d\sigma_Rds.
	\end{aligned}
\end{equation*} 
Recall that $\psi_n(x)=\psi(|x|/n)$. Since $\psi_{x^i}, \psi_{x^ix^j}$, and $\sup_{s\leq \tau,x\in\dR^d} \Psi_{a}(x)u(s,x)$ are bounded, and $e^{-(K_1-a)M(x)}$ has exponential decay as $|x|\to\infty$, the integral \eqref{int frA} is finite. 
Note that for any arbitrary stopping time $\kappa \leq \tau$, we have $(\fm^n_\kappa)_- \leq \fA^n_\kappa + \fB_\kappa $ from \eqref{eq:nonnegativity of sum}. Since $\fm^n_t$ is a local martingale, for any fixed $n$, one can choose a sequence of stopping times $\{\tau^i\}_i$ such that $\tau^i\rightarrow \infty$ and $\fm^n_{t \wedge\tau^i}$ is a martingale on $[0,T]$ for each $i\in \dN$. Therefore, $\fm^n_{t\wedge\tau^i \wedge \kappa }$ is a martingale by the optional sampling theorem, and thus implies that $\dE[\fm^n_{t\wedge \tau^i \wedge \kappa}]=0$ for any $t\leq T$. This ensures that $\dE[(\fm^n_{t\wedge \tau^i \wedge \kappa})_-] = \dE[(\fm^n_{t\wedge \tau^i \wedge \kappa})_+]$. Thus, we have 
\begin{equation*}
\dE[|\fm^n_{ \tau^i \wedge \kappa} |] \leq 2 \dE[\fA^n_\kappa + \fB_\kappa],
\end{equation*} recalling that $\fA^n_t+\fB_t$ is nondecreasing in $t$. Therefore, Fatou's lemma shows that 
\begin{equation*}
	\dE[|\fm^n_{ \kappa} |] \leq 2 \dE[\fA^n_{\kappa}+\fB_{\kappa}].
\end{equation*} Note that the above inequality holds for an arbitrary stopping time $\kappa\leq \tau$. Since $\fA^n_{\kappa}$ is nonnegative, by It\^o's inequality (e.g. \cite[Theorem III.6.8]{diffusion}), for $\alpha\in(0,1)$, we have
\begin{equation}
	\begin{aligned}
		\dE \left[\sup_{t\leq \tau} |\fm^n_t|^\alpha \right] \leq 2^\alpha \frac{2-\alpha}{1-\alpha} \dE\left[ \sup_{t\leq \tau} | \fA^n_t + \fB_t|^\alpha \right].
	\end{aligned}
\end{equation} 
Again recall that $\fA^n_t$ and $\fB_t$ are nondecreasing in $t$, and $\sup_{t\leq \tau, x\in \dR^d}\Psi_{a}(x)u(t,x)$ is bounded almost surely. These show that 
\begin{equation}
	\limsup_{n\rightarrow \infty} \dE\left[ \sup_{t\leq \tau} | \fA^n_t + \fB_t|^\alpha \right] \leq \dE \left [ |\fB_\tau|^\alpha \right].
\end{equation} On the other hand, by the Burkholder-Davis-Gundy inequality (e.g. \cite[Theorem IV.4.1]{diffusion}) with \eqref{condition of h} and passing to the limit as $n\rightarrow \infty$, we can see that
\begin{equation*}
\begin{aligned}
&\dE\left[\left( \int_0^\tau\int_{Q_R}\int_{Q_R} e^{-2K_2s-K_1(M(x+y)+M(y))}M(x+y)M(y)(u(s,x+y)u(s,y))^\lambda  dyf(dx)ds \right)^{\alpha/2} \right]\\
&\quad\leq N \cdot \limsup_{n\rightarrow \infty}\dE \left[\sup_{t\leq \tau} |\fm^n_t|^\alpha \right]\leq N\dE \left [ |\fB_\tau|^\alpha \right],
\end{aligned}
\end{equation*}
 where $N = N(\alpha,d,K)$.
This completes the proof of \eqref{eq:estimation of quadratic variation}.

\end{proof}

Next we present integral inequalities. In the proof of Lemma \ref{ito_prop}, these inequalities are used to adjust the exponent of integrals.

\begin{lemma}\label{lemma:integral inequality}

Let $\gamma,\lambda\in(0,1)$, $0\leq a<b<\infty$, and $H>1$. 
\begin{enumerate}[(i)]

\item  \label{item:integral inequality 1}
Suppose $g: \dR^d \to \dR$ is a bounded, nonnegative, and  continuous function such that
\begin{equation}
\label{holder regularity condi for reverse jensen}
g(x)\leq H \quad\text{and}\quad |g(x) - g(y)| \leq H|x-y|^\gamma\quad\text{for}\quad x,y\in \dR^d.
\end{equation}
Then, for $R>1$ and $r>0$ satisfying
\begin{equation}
\label{condition of r for integral inequality}
0 < r < \frac{R}{b}\wedge \left[ \frac{1}{b-a} \left(2^{-\frac{(\gamma+1)}{\gamma}}H^{-\frac{d}{\gamma}} R^{d-1} \left( \frac{d\pi^{d/2}}{\Gamma(d/2+1)} \right) \right)^{\frac{1}{\gamma+d-1}}\right],
\end{equation}
we have 
\begin{equation}\label{eq:integral inequality w.r.t. x}
\begin{aligned}
&\left( \int_{R+ar<|x|<R+br}g(x)dx \right)^{\frac{\gamma\lambda+d}{\gamma+d}} \leq N  R^{\frac{d(d-1)}{\gamma+d}} (r(b-a))^{\frac{-d(\gamma+d-1)}{\gamma+d}} \int_{R+ar<|x|<R+br}(g(x))^\lambda dx,
\end{aligned}
\end{equation}
where $N = N(\gamma,\lambda,H)$.

\item\label{item:integral inequality 2}
Let $T\in(0,\infty)$. Suppose $g:\dR_+ \to \dR$ is a nonnegative and  continuous function on $\dR$ such that
$$ |g(t) - g(s)| \leq H|t-s|^\gamma\quad\text{for}\quad t,s\in (0,\infty).
$$
If $g(0) = 0$, we have 
\begin{equation}\label{eq:integral inequality w.r.t. t}
\left( \int_0^T g(t)dt \right)^{\frac{\gamma\lambda+1}{\gamma+1}} \leq N  H^{1/\gamma}  \int_0^T (g(t))^\lambda \,dt,
\end{equation}
\end{enumerate}
where $N = N(\gamma,\lambda)$.
\end{lemma}

\begin{proof}
The idea of proof follows from  \cite[Theorem 4]{burdzy2010nonuniqueness} and \cite[Lemma 10.4.2]{rippl2013pathwise}.
    To prove Lemma~\ref{lemma:integral inequality} \eqref{item:integral inequality 1}, we separate the proof into two steps.

\textbf{ (Step 1)} In this step, we consider the case
\begin{equation}
\label{eq:first assumption of the proof of integral estimate}
\int_{R+ar<|x|<R+br} g(x)\,dx = 1.
\end{equation}
First we assume $\sup_{ R+ar<|x|<R+br }g(x)<1$. Notice that $|g(x)| < |g(x)|^\lambda$ for $x\in \{ x : R+ar<|x|<R+br \}$. Therefore, \eqref{eq:first assumption of the proof of integral estimate} yields
\begin{equation}\label{eq:integral inequality case 1}
\begin{aligned}
\int_{ R+ar<|x|<R+br } (g(x))^\lambda\, dx \geq 1.
\end{aligned}
\end{equation}

Now we consider the case where  $\sup_{R+ar<|x|<R+br}g(x)\geq1$. Due to the continuity of $g$, there exists $x_0\in \{ x : R+ar<|x|<R+br \}$ such that $g(x_0)>1/2 $. Then, since $g$ satisfies \eqref{holder regularity condi for reverse jensen}, for any $x\in\{ x : R+ar<|x|<R+br \}$ with $|x-x_0|\leq (4H)^{-1/\gamma}$, we have 
\begin{equation*}
g(x) \geq g(x_0) - |g(x) - g(x_0)| \geq \frac{1}{2} - H|x-x_0|^{\gamma}\geq \frac{1}{2} - H(4H)^{-1} = 1/4
\end{equation*}
Therefore, we have
\begin{equation}\label{eq:integral inequality case 2}
\begin{aligned}
\int_{ R+ar<|x|<R+br } (g(x))^\lambda\,dx
&\geq \int_{ R+ar<|x|<R+br } (g(x))^\lambda\bi_{\{|x-x_0|<(4H)^{-1/\gamma}\}}\,dx \\
&\geq 4^{-\lambda}\int_{ R+ar<|x|<R+br }\bi_{\{|x-x_0|<(4H)^{-1/\gamma}\}}\,dx \\
&\geq 4^{-\lambda}\prod_{i}\int_{ \frac{R+ar}{\sqrt{d}}<|x^i|<\frac{R+br}{\sqrt{d}} }\bi_{\{|x^i-x_0^i|<\frac{(4H)^{-1/\gamma}}{\sqrt{d}}\}}\,dx^i \\
&\geq N(\lambda,d) \left(  (4H)^{-1/\gamma} \wedge (b-a)r\right)^{d}.
\end{aligned}
\end{equation}
The last inequality follows from \cite[Lemma 10.4.1]{rippl2013pathwise}. Combining \eqref{eq:integral inequality case 1} and \eqref{eq:integral inequality case 2}, we get 
\begin{equation}\label{eq:integral inequality step 1}
\int_{ R+ar<|x|<R+br } (g(x))^\lambda\,dx \geq N\left(  (4H)^{-1/\gamma} \wedge (b-a)r\right)^{d},
\end{equation}
where $N = N(\lambda,d)$ is a positive constant.

\textbf{ (Step 2)} For general $g$, we set 
\begin{equation}
\label{def of G}
G := \left( \int_{R+ar<|x|<R+br} g(x)\,dx \right)^{\frac{1}{\gamma+d}}\quad\text{and}\quad\cG(x) := G^{-\gamma}g(Gx).
\end{equation} Then, by the change of variable, we have 
\begin{equation*}
\begin{aligned}
\int_{\frac{R+ar}{G}<|x|<\frac{R+br}{G}} \cG(x) \, dx 
&= G^{-(\gamma+d)}\int_{R+ar<|x|<R+br} g(x)\,  dx = 1.
\end{aligned}
\end{equation*}
Furthermore, since 
$$ |\cG(x) - \cG(y)| \leq G^{-\gamma}|g(Gx) - g(Gy)| \leq HG^{-\gamma}|Gx-Gy|^\gamma = H|x-y|^{\gamma},$$
\eqref{eq:integral inequality step 1} yields
\begin{equation}\label{eq:integral inequality general case}
\begin{aligned}
N(\lambda,d)\left(  (4H)^{-1/\gamma} \wedge \frac{(b-a)r}{G}\right)^{d}
&\leq\int_{\frac{R+ar}{G}<|x|<\frac{R+br}{G}} (\cG(x))^\lambda \, dx \\
&= G^{-\gamma\lambda-d}\int_{R+ar<|x|<R+br}(g(x))^\lambda \,dx.
\end{aligned}
\end{equation} 
Observe that
\begin{equation}
\begin{aligned}\label{eq:bound of A}
\frac{G}{r(b-a)}
&= \frac{1}{r(b-a)}\left( \int_{R+ar < |x| < R+br} g(x)dx \right)^{\frac{1}{\gamma+d}} \\
&\leq \frac{1}{r(b-a)}\left( \frac{\pi^{d/2}}{\Gamma(d/2+1)} \right)^{\frac{1}{\gamma+d}} \left( (R+br)^d - (R+ar)^d \right)^{\frac{1}{\gamma+d}}H^{\frac{1}{\gamma+d}}\\
&\leq \left( \frac{d\,H\,2^{d-1}\,\pi^{d/2}}{\Gamma(d/2+1)}  \right)^{\frac{1}{\gamma+d}} R^{\frac{d-1}{\gamma+d}} \left( r(b-a) \right)^{-\frac{\gamma+d-1}{\gamma+d}}.
\end{aligned}
\end{equation}
Also, note that for any $a,b>0$, if $a \wedge b  \leq N_0$ for some $N_0\in(0,\infty)$, then 
\begin{equation}
\label{min max ineq}
1 \leq N_0\left( \frac{1}{a} \vee \frac{1}{b}  \right) \leq N_0\left( \frac{1}{a} + \frac{1}{b}  \right).
\end{equation}
Thus, since $R,H>1$ and $r$ satisfies \eqref{condition of r for integral inequality}, if we multiply $G^{\gamma\lambda+d}$ both sides of \eqref{eq:integral inequality general case}, \eqref{eq:bound of A} and \eqref{min max ineq} imply
\begin{equation}
\label{int ineq}
\begin{aligned}
&\left( \int_{R+ar<|x|<R+br} g(x)\,dx \right)^{\frac{\gamma \lambda +d}{\gamma+d}}\\
&\leq N\left(  (4H)^{1/\gamma} \vee \frac{G}{r(b-a)}\right)^{d} \int_{R+ar<|x|<R+br} (g(x))^\lambda \, dx\\
&\leq N\left(  (4H)^{1/\gamma} + R^{\frac{d-1}{\gamma+d}}\left( r(b-a) \right)^{-\frac{\gamma+d-1}{\gamma+d}}\left( \frac{d\,H\,2^{d-1}\,\pi^{d/2}}{\Gamma(d/2+1)}  \right)^{\frac{1}{\gamma+d}}  \right)^{d} \int_{R+ar<|x|<R+br} (g(x))^\lambda \, dx\\
&\leq N R^{\frac{d(d-1)}{\gamma+d}} (r(b-a))^{\frac{-d(\gamma+d-1)}{\gamma+d}} \int_{R+ar<|x|<R+br} (g(x))^\lambda \, dx,
\end{aligned}
\end{equation}
where $N = N(\gamma,\lambda,d, H)$.
Thus, Lemma~\ref{lemma:integral inequality} \eqref{item:integral inequality 1} is proved.

Finally, we prove Lemma \ref{lemma:integral inequality} \eqref{item:integral inequality 2}. Since the proof is similar to the one of Lemma \ref{lemma:integral inequality} \eqref{item:integral inequality 1} with $d = 1$, we only point out the differences. Instead of \eqref{int ineq}, we have
\begin{equation}
\label{proof of integral ineq in t}
\left( \int_0^T  g(t) \,dt \right)^{\frac{\gamma \lambda +1}{\gamma +1}} 
\leq N\left( (4H)^{1/\gamma} \vee \frac{G}{T}\right)\int_0^T (g(t))^\lambda \,dt,
\end{equation} 
where $G:=\left(\int_0^T g(t) \, dt\right)^{1/(\gamma+1)}$. Since we assume $g(0) = 0$, we have $g(t)= g(t)-g(0)\leq H|t-0|^\gamma \leq H T^\gamma$ for all $t\in(0,T)$, and thus $G\leq H^{1/(\gamma+1)} T$. Therefore, from \eqref{proof of integral ineq in t} and \eqref{min max ineq},
\begin{equation*}
\left( \int_0^T  g(t) \,dt \right)^{\frac{\gamma \lambda +1}{\gamma +1}} \leq N\left( (8H)^{\frac{1}{\gamma}} + H^{\frac{1}{\gamma+1}}\right)\int_0^T (g(t))^\lambda \,dt,
\end{equation*} 
which completes the proof. 
 \end{proof}

\vspace{2mm}

Now we introduce an integral inequality that produces  a lower bound of correlation measure $f$ in \eqref{def of covaiance functional}. The motivation of the proof follows from \cite{CH2019,CK2019}.

\begin{lemma}\label{lemma:lower bound of covariance} 
For the given correlation measure $f$, there exists a nonnegative function $\varphi \in C^\infty_c(\dR^d)$ with $\int_{\dR^d} \varphi(x) dx >0$ such that for any bounded nonnegative measurable function $g$ with compact support on $\dR^d$, we have
	\begin{equation}
	\label{lower bound ineq}
		\int_{\dR^d} (g\ast\tilde g)(x) f(dx) \geq  \int_{\dR^d} |(g \ast \varphi)(x)|^2 dx,
	\end{equation}
	where $\tilde g(x) = g(-x)$.
\end{lemma}
\begin{proof}
Let us  first  construct $\varphi$ by approximating $f$ with a continuous nonnegative function. Define 
\begin{equation*}
\psi(x) : = 4^{-d} \prod_{i=1}^d (2-|x^i|)\bi_{|x^i|\leq 2}(x).
\end{equation*}
Observe that $\psi$ is a continuous, nonnegative and nonnegative definite function on $\dR^d$. In addition, $\psi$ has a compact support and $\int_{\dR^d} \psi(x) dx=1.$ For $\ep>0$, define $\psi_\ep(x) : = \ep^{-d}\psi(x/\ep)$. Then, we have
\begin{equation}
\label{eq:upper bound of fourier transform of mollifier}
|\hat{\psi}_\ep(\xi)|^2 = |\hat{\psi}(\ep \xi)|^2 = \left| \int_{\dR^d} e^{-i\ep\langle \xi,x\rangle } \psi(x) dx\right|^2 \leq \left|\int_{\dR^d}  \psi(x) dx\right|^2 =1,
\end{equation} where $\hat{\psi}(\xi) : = \cF(\psi)(\xi)$.
With given $\psi_\ep$, set $f_{\ep}(x) : = (\psi_\ep\ast\psi_\ep\ast f)(x)$. Then,  $f_\ep$ is a nonnegative and nonnegative definite $C^2(\dR^d)$ \textit{function} that is not identically 0. Thus,  $f_\ep(0) >0$  since $f_\ep$ achieves its maximum at the origin (thanks to the fact that $f_\ep$ is  a nonnegative definite function). Thus, there exists a constant $r>0$ satisfying $\inf_{|x|\leq r} f_\ep(x) \geq c>0$ for some $c>0$, which implies that  there exists $\varphi\in C^\infty_c(\dR^d)$ such that $\varphi\geq 0$, $supp(\varphi) \subseteq \{|x| \leq r\}$, and $0< \sup_{x\in\dR^d}(\varphi\ast\tilde{\varphi})(x) = \lVert \varphi\rVert_{L_2} \leq c/2$.

Next, we approximate $g$ by $C_c(\dR^d)$ functions to apply \eqref{def of covaiance functional}. Choose $\zeta\in C^\infty_c(\dR^d)$ such that $\int_{\dR^d} \zeta(x) dx =1$ and set  $\zeta_\delta(x):= \delta^{-d} \zeta(x/\delta)$ with $\delta>0$. Recall $g^{(\delta)}(x)= (g\ast \zeta_\delta)(x)$ in \eqref{def of mollification}. Since $g^{(\delta)}(\cdot)\in C_c(\dR^d)$, by Fatou's lemma and \eqref{def of covaiance functional}, we have  
\begin{equation}
\label{proof of lower bound ineq}
	\begin{aligned}
		\int_{\dR^d} |(g\ast \varphi)(x)|^2 dx & = \int_{\dR^d}\int_{\dR^d} g(x) g(y) (\varphi\ast\tilde{\varphi})(x-y) dxdy \\
		&\leq   \int_{\dR^d}\int_{\dR^d} g(x) g(y) f_\ep(x-y) dxdy \\
		&\leq \liminf_{\delta\downarrow 0}  \int_{\dR^d}\int_{\dR^d} g^{(\delta)}(x) g^{(\delta)}(y) f_\ep(x-y) dxdy\\
		&= \liminf_{\delta\downarrow 0} \int_{\dR^d} \left(g^{(\delta)}\ast \tilde{g}^{(\delta)}\right)(x) f_\ep(x) dx \\
		&= \liminf_{\delta\downarrow 0} \int_{\dR^d} \left| \hat{g}^{(\delta)}(\xi) \right|^2 \hat{f}_\ep(\xi) d\xi\\
		&= \liminf_{\delta\downarrow 0} \int_{\dR^d} \left| \hat{g}^{(\delta)}(\xi)\right|^2 \left|\hat{\psi}_\ep(\xi) \right|^2 \hat{f}(d\xi).
	\end{aligned} 
\end{equation} 
The last equality follows from $f_\ep: = \psi_\ep\ast\psi_\ep\ast f$. By applying \eqref{eq:upper bound of fourier transform of mollifier} to \eqref{proof of lower bound ineq}, we have
\begin{equation*}
	\begin{aligned}
		\int_{\dR^d} |(g\ast \varphi)(x)|^2 dx &\leq \liminf_{\delta\downarrow 0} \int_{\dR^d} \left| \hat{g}^{(\delta)}(\xi) \right|^2 \hat{f}(d\xi)\\
		&= \liminf_{\delta\downarrow 0}\int_{\dR^d} \left( g^{(\delta)}\ast\tilde g^{(\delta)} \right)(x) f(dx)\\
		&= \int_{\dR^d} (g\ast\tilde g)(x) f(dx),
	\end{aligned}
\end{equation*} 
where the last equality is justified by the bounded convergence theorem. This completes the proof. 
\end{proof}

We are now ready to prove Lemma \ref{ito_prop}.
\begin{proof}[\textbf{Proof of Lemma \ref{ito_prop}}]

First we prove \eqref{key_prop_1}. Let $\fp,\fq>0, \alpha \in (0,1)$, and recall $ l = \frac{\gamma\lambda+d}{\gamma+d}\in(0,1)$, where $\gamma$ is the constant introduced in Theorem \ref{thm:existence of solution}. Define
\begin{equation}
\label{definition of kappa}
\kappa := \tau\wedge\inf\left\{ t\geq0: \int_0^t \left( \int_{\partial Q_R}  u(s,\sigma) d\sigma_\sigma \right)^{ l} ds \geq \fq^ l \right\},
\end{equation}
and 
\begin{equation}
	U_R(s,x) := e^{-K_2s}e^{-K_1M(x)}M(x)(u(s,x))^\lambda,
\end{equation}
where $M(x)$ is the function introduced in \eqref{def of M(x)}.
By Chebyshev's inequality, \eqref{eq:estimation of quadratic variation}, and the definition of $\kappa$, we have
\begin{equation}
\label{ineq to prove maximum principle}
\begin{aligned}
&\dP\left( \int_0^\tau \int_{\dR^d} (U_R(s,\cdot)\ast \tilde U_R(s,\cdot))(x)  f(dx)ds \geq \fp^ l \right) \\
&\leq  \dP\left( \kappa < \tau \right) + \dP\left( \int_0^\kappa \int_{\dR^d} (U_R(s,\cdot)\ast \tilde U_R(s,\cdot))(x) f(dx) ds \geq \fp^ l \right) \\
&\leq \dP\left( \int_0^\tau \left( \int_{\partial Q_R}  u(s,\sigma) d\sigma_R \right)^{ l} ds \geq \fq^ l \right) + \fp^{- l\alpha/2}N\dE\left[\left(\int_0^\kappa \int_{\partial Q_R} u(s,\sigma) d\sigma_R ds\right)^\alpha\right],
\end{aligned}
\end{equation}
where $N = N(\alpha,d,K)$. Now we control the last term of \eqref{ineq to prove maximum principle}. 
 Due to \eqref{holder reg for ito ineq}, we have
\begin{equation*}
\begin{aligned}
\left| \int_{\partial Q_R} u(t,\sigma)d\sigma_R - \int_{\partial Q_R} u(s,\sigma) d\sigma_R \right| \leq \int_{\partial Q_R} \left| u(t,\sigma) - u(s,\sigma)\right| d\sigma_R  \leq N(d) R^{d-1} e^{aR}H  |t-s|^\gamma.
\end{aligned}
\end{equation*}
for $t,s\in(0,T)$. Additionally, note that $\int_{\partial Q_R} u(0,\sigma) d\sigma_R = \int_{\partial Q_R} u_0(\sigma) d\sigma_R = 0$. Therefore, Lemma \ref{lemma:integral inequality} \eqref{item:integral inequality 2} implies
\begin{equation}
\label{apply integral ineq in t}
\int_0^{\kappa} \int_{ \partial Q_R } u(s,\sigma) d\sigma_R ds \leq N\left(\int_0^{\kappa} \left( \int_{ \partial Q_R } u(s,\sigma) d\sigma_R \right)^ l ds \right)^{\frac{\gamma+1}{\gamma l+1}},
\end{equation}
where $N = N(\gamma,a,d,R,H)$. Thus, by definition of $\kappa$ and applying \eqref{apply integral ineq in t} to the last term of \eqref{ineq to prove maximum principle}, we have
\begin{equation}
\label{in the pf of maximum prin, controlling the second term}
\begin{aligned}
\dE\left[\left( \int_0^{\kappa} \int_{ \partial Q_R } u(s,\sigma) d\sigma_R ds \right)^{\alpha} \right]\leq N\dE\left[\left( \int_0^{\kappa} \left( \int_{ \partial Q_R } u(s,\sigma) d\sigma_R \right)^ l ds \right)^{\frac{\gamma+1}{\gamma l+1}\alpha} \right]\leq N\fq^{\frac{(\gamma+1) l\alpha}{\gamma l+1}},
\end{aligned}
\end{equation}
where $N = N(\alpha,\gamma,a,d,H,R)$. Therefore, by applying \eqref{in the pf of maximum prin, controlling the second term} to the last term of \eqref{ineq to prove maximum principle},
\begin{equation}
\label{before letting p,q goes to zero}
\begin{aligned}
&\dP\left( \int_0^\tau \int_{\dR^d} (U_R(s,\cdot)\ast \tilde U_R(s,\cdot))(x) f(dx) ds \geq \fp^ l \right) \\
&\quad\quad\leq \dP\left( \int_0^\tau \left( \int_{\partial Q_R}  u(s,\sigma) d\sigma_R \right)^{ l} ds \geq \fq^ l \right) + N \left(\frac{\fq^{\frac{\gamma+1}{\gamma l+1}}}{\fp^{\frac{1}{2}}}\right)^{ l\alpha},
\end{aligned}
\end{equation}
where $N = N(\alpha,\gamma,\lambda,a,d,R,H,K)$. By letting $\fq\downarrow0$ and $\fp\downarrow0$ in order, \eqref{before letting p,q goes to zero} implies
\begin{equation}
\label{proof of ito prop first assertion}
\begin{aligned}
&\dP\left( \int_0^\tau \left( \int_{\partial Q_R}  u(s,\sigma) d\sigma_R \right)^ l ds = 0 \right) 
\leq \dP\left( \int_0^\tau \int_{\dR^d} (U_R(s,\cdot)\ast \tilde U_R(s,\cdot))(x) f(dx)  ds = 0 \right).
\end{aligned}
\end{equation}
Next, we claim that
\begin{equation}
\label{proof of ito prop first assertion 2}
	\begin{aligned}
		&\left\{ \omega\in\Omega :  \int_0^\tau \int_{\dR^d} (U_R(s,\cdot)\ast \tilde U_R(s,\cdot))(x) f(dx)  ds = 0 \right\} \\
		&\quad \subset \left\{ \omega\in\Omega :  u(s,x) =0 \quad \text{for all }s\in[0,\tau] \text{ and } x\in Q_R \right\}.
	\end{aligned}
\end{equation}
Indeed, for $S>R$, let us set
\begin{equation*}
	g_{R,S}(s,x) :=  M(x)e^{-K_1M(x)}(u(s,x))^\lambda\bi_{Q_{R}\setminus Q_{S}}(x)
\end{equation*} 
and choose $\varphi \in C^\infty_c(\dR^d)$ introduced in Lemma \ref{lemma:lower bound of covariance}. Since $g_{R,S}$ is a bounded nonnegative measurable function with compact support on $\dR^d$, we have
\begin{equation}
\label{cut-off decomposition}
\begin{aligned}
\int_{\dR^d} \left( g_{R,S}(s,\cdot)\ast \tilde g_{R,S}(s,\cdot) \right)(x)f(dx) &\geq \int_{\dR^d} |(g_{R,S}(s,\cdot) \ast \varphi)(x)|^2 d x \\
&= \int_{\dR^d}\int_{\dR^d} g_{R,S}(s,x)g_{R,S}(s,y) (\varphi\ast \tilde{\varphi})(x-y)dydx.
\end{aligned}
\end{equation} 
Therefore, if we take $\omega\in\Omega$ satisfying $\int_0^\tau\int_{\dR^d} \left( U_{R}(s,\cdot)\ast \tilde U_{R}(s,\cdot) \right)(x)f(dx)ds=0$, then $U_R \geq g_{R,S}$, nonnegativity of $f$, and \eqref{cut-off decomposition} yield
$$\int_0^\tau\int_{Q_R\setminus Q_S}\int_{Q_R\setminus Q_S} e^{-2K_2s-K_1(M(x)+M(y))}M(x)M(y)(u(t,x))^\lambda(u(t,y))^\lambda (\varphi\ast \tilde{\varphi})(x-y) dydxds  = 0.
$$
Thus,
$$ (u(s,x))^\lambda(u(s,y))^\lambda (\varphi\ast\tilde \varphi)(x-y) = 0
$$ 
almost every $(s,x,y)\in(0,\tau)\times (Q_R\setminus Q_S)\times(Q_R\setminus Q_S)$. Since $ (\varphi\ast\tilde \varphi)(0)>0$, we have $(u(t,x))^{2\lambda} = 0$ for almost every $(t,x)\in(0,\tau)\times (Q_R\setminus Q_S)$. Since $S>R$ is arbitrary, the continuity of $u$ implies \eqref{proof of ito prop first assertion 2}. Then, by \eqref{proof of ito prop first assertion} and \eqref{proof of ito prop first assertion 2}, we have \eqref{key_prop_1}. \\

Next we prove \eqref{key_prop_2}.  Let $\alpha \in (0,1)$, $\fp,\fq>0$, and $ \fr>0$ satisfy \eqref{condition of r for decay prop}. Again, we set 
$$ \kappa := \tau\wedge\inf\left\{ t\geq0: \int_0^t \left( \int_{\partial Q_R} u(s,\sigma) d\sigma_R \right)^ l ds \geq \fq^ l \right\}.
$$
Then, by Chebyshev's inequality and Jensen's inequality, we have
\begin{equation}
\label{eq: upper bound using kappa}
\begin{aligned}
&\fr^{-1}\int_\fr^{2\fr} \dP\left( \int_0^\tau \left( \int_{\partial Q_{R+z}} u(s,\sigma) d\sigma_{R+z} \right)^ l ds \geq \fp^ l \right)dz \\
&\quad\leq  \dP(\kappa<\tau) + \fr^{-1}\int_\fr^{2\fr} \dP\left( \int_0^{\kappa}  \left( \int_{\partial Q_{R+z}} u(s,\sigma) d\sigma_{R+z} \right)^ l ds \geq \fp^ l \right)dz \\
&\quad\leq \dP\left( \int_0^\tau \left( \int_{\partial Q_{R}} u(s,\sigma) d\sigma_{R} \right)^ l ds \geq \fq^ l \right) \\
&\quad\quad + \fr^{-1}\fp^{-\alpha l}\int_\fr^{2\fr} \dE\left[\left( \int_0^{\kappa}  \left( \int_{\partial Q_{R+z}} u(s,\sigma) d\sigma_{R+z} \right)^ l  ds  \right)^\alpha\right] dz.
\end{aligned} 
\end{equation}
To bound the last term of \eqref{eq: upper bound using kappa}, we use Jensen's inequality to get 
\begin{equation}
\label{applying int by parts}
\begin{aligned}
&\int_\fr^{2\fr} \dE\left[\left( \int_0^{\kappa}  \left( \int_{\partial Q_{R+z}} u(s,\sigma) d\sigma_{R+z} \right)^ l  ds  \right)^\alpha \right]dz \\
&\quad \leq \fr^{1-\alpha l} \dE\left[\left( \int_0^{\kappa} \left( \int_\fr^{2\fr} \int_{\partial Q_{R+z}} u(s,\sigma) d\sigma_{R+z}dz \right)^ l  ds  \right)^\alpha \right] \\
&\quad = \fr^{1-\alpha l} \dE\left[\left( \int_0^{\kappa} \left( \int_{R+\fr<|x|<R+2\fr} u(s,x) dx  \right)^ l  ds  \right)^\alpha \right].
\end{aligned}
\end{equation}
Thanks  to \eqref{holder reg for ito ineq}, we can use  Lemma \ref{lemma:integral inequality}  \eqref{item:integral inequality 1} to see that  the last term of \eqref{applying int by parts} is dominated by
\begin{equation}
\label{apply integral ineq in x}
\begin{aligned}
&\dE\left[\left( \int_0^{\kappa}  \left( \int_{R+\fr < |x| < R+2\fr}  u(s,x) dx\right)^ l   ds \right)^{\alpha}\right] \\
&\quad \leq Ne^{\alpha l aR}\dE\left[\left( \int_0^{\kappa}  \left( \int_{R+\fr < |x| < R+2\fr} \Psi_a(x) u(s,x) dx\right)^ l   ds \right)^{\alpha}\right] \\
&\quad\leq N R^{\frac{\alpha d(d-1)}{\gamma+d}} e^{\alpha (l-\lambda) a R} \fr^{\frac{-\alpha d(\gamma+d-1)}{\gamma+d}}\dE\left[\left( \int_0^{\kappa} \int_{R+\fr < |x| < R+2\fr} (u(s,x))^\lambda dx   ds \right)^{\alpha}\right],
\end{aligned}
\end{equation}
where $l = \frac{\gamma\lambda+d}{\gamma+d}$ and $N = N(\alpha,\gamma,\lambda,a,d,H)$. Since the difference between $Q_{R+\fr}\setminus Q_{R+2\fr}$ and $ \{x:R+\fr < |x| < R+2\fr\}$ has measure zero, by combining \eqref{applying int by parts} and \eqref{apply integral ineq in x}, we have
\begin{equation}
\label{ineq:key ineq in the proof of key lemma 2}
\begin{aligned}
&\fr^{-1}\int_\fr^{2\fr}\dE\left[\left( \int_0^{\kappa}  \left( \int_{\partial Q_{R+z}} u(s,\sigma) d\sigma_{R+z} \right)^ l  ds  \right)^\alpha \right]dz\\
&\quad \leq NR^{\frac{\alpha d(d-1)}{\gamma+d}} e^{\alpha (l-\lambda) a R} \fr^{\frac{-\alpha d(\gamma+d-1)}{\gamma+d}-\alpha  l}\dE\left[\left( \int_0^{\kappa} \int_{Q_{R+\fr} \setminus Q_{R+2\fr}} \left( u(s,x) \right)^\lambda dx ds \right)^{\alpha}\right],
\end{aligned}
\end{equation}
where $N = N(\alpha,\gamma,\lambda,a,d,H,K)$.

To proceed further, we set 
$$g(s,x) := M(x)e^{-K_1M(x)}(u(s,x))^\lambda\bi_{Q_{R+\fr}\setminus Q_{R+2\fr}}(x).
$$ and let $\varphi\in C_c(\dR^d)$ be the function introduced   in Lemma  \ref{lemma:lower bound of covariance}. Without loss of generality, we may assume that $supp(\varphi) \subseteq Q_R^c$. Then, note that
\begin{equation}
\label{integral ineq of a function with compact support}
\begin{aligned}
\left( \int_{\dR^d} g(s,x) dx \right)^2 &\leq N\left( \int_{\dR^d} (g(s,\cdot)\ast\varphi)(x) dx \right)^2  \\
&= N\left( \int_{|x|<2(R+2\fr)} (g(s,\cdot)\ast\varphi)(x) dx \right)^2  \\
&\leq N R^{d}\int_{|x|<2(R+2\fr)} \left| (g(s,\cdot)\ast\varphi)(x) \right|^2 dx,
\end{aligned}
\end{equation}
where $N = N(d,f)$.
The equality in the second line holds since $supp(g(s,\cdot))+supp(\phi) \subseteq \{|x| <2(R+2\fr)\}.$ Thus, by applying Lemma \ref{lemma:lower bound of covariance} to \eqref{integral ineq of a function with compact support}, we have
\begin{equation}
\label{ineq:key ineq in the proof of key lemma 4}
\begin{aligned}
\left(\int_{\dR^d} g(s,x) dx \right)^2
\leq NR^{d} \int_{\dR^d} (g(s,\cdot)\ast\tilde g(s,\cdot))(x) f(dx),
\end{aligned}
\end{equation}
where $N = N(d,f)$. We now use  \eqref{ineq:key ineq in the proof of key lemma 4} and Jensen's inequality to get that 
\begin{equation}
\label{ineq:key ineq in the proof of key lemma 3}
\begin{aligned}
R^{-\alpha d/2}&\fr^{\alpha}e^{-2K_1\fr\alpha}\left( \int_0^\tau \int_{Q_{R+\fr}\setminus Q_{R+2\fr}}(u(s,x))^\lambda dx ds\right)^\alpha \\
&\leq R^{-\alpha d/2}\left( \int_0^\tau \int_{Q_{R+\fr}\setminus Q_{R+2\fr}} g(s,x) dx ds \right)^\alpha \\
&\leq N\left( \int_0^\tau \int_{\dR^d} e^{-2K_2s}\left(g(s,\cdot)\ast \tilde g(s,\cdot)\right)(x) f(dx) ds\right)^{\alpha/2},
\end{aligned}
\end{equation}
where $N = N(\alpha,d,f,K,T)$. Then, by combining \eqref{ineq:key ineq in the proof of key lemma 3} and Lemma \ref{prop:estimation of quadratic variation}, since $\fr \leq 1$, we have
\begin{equation}
\label{kernel_calculation}
\begin{aligned}
&\dE\left[\left( \int_0^{\kappa} \int_{Q_{R+\fr} \setminus Q_{R+2\fr}} \left(u(s,x) \right)^\lambda dx ds \right)^{\alpha}\right] \\
&\quad \leq N R^{\alpha d/2}  \fr^{-\alpha}  e^{ 2K_1 \fr \alpha }  \dE\left[\left( \int_0^{\kappa} \int_{ \partial Q_R }e^{-K_2s}u(s,\sigma) d\sigma_R ds \right)^{\alpha}\right], \\
&\quad \leq N R^{\alpha d/2}  \fr^{-\alpha}    \dE\left[\left( \int_0^{\kappa} \int_{ \partial Q_R }e^{-K_2s}u(s,\sigma) d\sigma_R ds \right)^{\alpha}\right], \\
\end{aligned}
\end{equation}
where $N = N(\alpha,\lambda,d,f,H,K,T)$. On the other hand,  \eqref{holder reg for ito ineq} implies
\begin{equation*}
\begin{aligned}
&\left| \int_{\partial Q_R} u(t,\sigma) d\sigma_R - \int_{\partial Q_R} u(s,\sigma) d\sigma_R \right| 
&\leq N(d)R^{d-1}  e^{aR} H  |t-s|^\gamma.
\end{aligned}
\end{equation*}
and $\int_{\partial Q_R} u(0,\sigma) d\sigma_R = 0$. Therefore, by Lemma \ref{lemma:integral inequality} \eqref{item:integral inequality 2}, we have
\begin{equation}
\label{apply def of kappa}
\begin{aligned}
\dE\left[\left( \int_0^{\kappa} \int_{ \partial Q_R } u(s,\sigma) d\sigma_R  ds \right)^{\alpha} \right]
&\leq N  R^{\frac{L \alpha (d-1)}{\gamma}}e^{\frac{\alpha a L R}{\gamma}} \dE\left[\left(\int_0^{\kappa} \left( \int_{ \partial Q_R } u(s,\sigma) d\sigma_R \right)^ l ds \right)^{L \alpha} \right]\\
&\leq N R^{\frac{L \alpha (d-1)}{\gamma}}e^{\frac{\alpha a L R}{\gamma}} \fq^{L\alpha l},
\end{aligned}
\end{equation}
where $L := \frac{\gamma+1}{\gamma l+1} := \frac{\gamma(\gamma+d)+\gamma+d}{\gamma(\gamma\lambda+d)+\gamma+d}$ and $N = N(\alpha,\gamma,\lambda,d,H)$.
Therefore, by combining \eqref{eq: upper bound using kappa}, \eqref{ineq:key ineq in the proof of key lemma 2}, \eqref{kernel_calculation}, and \eqref{apply def of kappa}, we have 
\begin{equation*}
\begin{aligned}
&\fr^{-1}\int_\fr^{2\fr} \dP\left( \int_0^\tau \left( \int_{\partial Q_{R+z}} u(s,\sigma) d\sigma_{R+z} \right)^ l ds \geq \fp^ l \right)dz \\
&\leq \dP\left( \int_0^\tau \left( \int_{\partial Q_{R}} u(s,\sigma) d\sigma_{R} \right)^ l ds \geq \fq^ l \right) \\
&\quad+ N   R^{\alpha\left( \frac{ d}{2}+\frac{ d (d-1)}{\gamma+d}+\frac{L (d-1)}{\gamma} \right)}  e^{\alpha  a R\left( l - \lambda + \frac{L}{\gamma} \right)} \fr^{ -\alpha\left( 1 +  l + \frac{ d (\gamma+d-1)}{\gamma+d}\right)}  \left(\frac{\fq^L }{\fp}\right)^{\alpha l},\\
\end{aligned}
\end{equation*}
where $L := \frac{\gamma+1}{\gamma l+1} := \frac{\gamma(\gamma+d)+\gamma+d}{\gamma(\gamma\lambda+d)+\gamma+d}$ and $N = N(\alpha,\gamma,\lambda,a,d,f,H,K,T)>0$. This implies \eqref{ito ineq} and we complete the proof.

\end{proof}

\section{Proof of Lemma \ref{L_1 bound with weight}}
\label{sec:proof of L_1 bound}

In this section, we  provide  the proof of Lemma \ref{L_1 bound with weight}. Throughout this section, we assume that there exist $a,H\in(0,\infty)$ satisfying \eqref{holder reg for ito ineq0}. To prove Lemma \ref{L_1 bound with weight}, we need an $L_1$ bound of the solution.


\begin{lemma}
\label{lemma:integral bound on solution with weight}
Let $\tau\leq T$ be a bounded stopping time.
Then, for every $R > (R_0\vee 1)+1$ and $\delta>0$, we have
\begin{equation}
\label{L1 bound with weight}
\begin{aligned}
\dE\left[\int_{Q_{R}}|x|e^{\delta |x|}u( \tau,x)dx\right] 
&\leq N
\end{aligned}
\end{equation}
where $N = N(\delta, a,d,m, H, K,R_0,T)$ and $R_0$ is the constant introduced in Assumption \ref{assumption on initial data}. In particular,
\begin{equation}
\label{L1 bound with weight - limit}
	\lim_{R\to\infty}\dE\left[\int_{Q_{R}}|x|e^{\delta |x|}u( \tau ,x)dx\right] = 0.
\end{equation}
\end{lemma}
\begin{proof}

For convenience, set
$$q(x) := |x|e^{\delta|x|}
$$
on $\dR^d$. Choose $R_1\in(R_0\vee1,(R_0\vee1)+1)$.
Let $\psi\in C_c^\infty(\dR)$ be a nonnegative symmetric function satisfying $\psi(z) = 1$ on $|z|<1$ and $\psi(z) = 0$ on $|z|\geq2$ and define $\psi_n(x) := \psi(|x|/n)$ for $n\in \dN$. Take a nonnegative function $\zeta\in C_c^\infty(\dR)$ satisfying $\int_{\dR}\zeta(z)dz = 1$, symmetric, and $\zeta(z) = 0$ on $|z|\geq1$. For $m\in \dN$, define
$$\varrho_{m} (x) := m\int_{\dR} \bi_{|z|>R_1 - \frac{1}{m}} (z)\zeta \left(m \left(|x| - z\right)\right) dz = \int_{\dR} \bi_{\left| |x|-\frac{1}{m}z \right|>R_1 - \frac{1}{m}}(z) \zeta \left(z\right) dz.
$$ 
Additionally, for $g\in L_{1,loc}(\dR^d)$ and $\ep>0$, set $g^{(\ep)}(x)$ as in \eqref{def of mollification}. 
Choose $K_0>0$ such that
\begin{equation}
\label{def of K_0}
K_0 > (5 + 6\delta + \delta^2)K
\end{equation}
where $K$ is the constant introduced in Assumption \ref{assumptions on coefficients}.
Then, for any $t\leq T$ and large $m\in\dN$, since $supp(u_0)\subset  \{ |x| \leq R_0 \}$ and $R_0<R_1$,   we have 
\begin{equation}
\label{calculation to obtain L1 bound 1}
\begin{aligned}
\dE &\left[\int_{Q_{R_1}}e^{-K_0(t\wedge \tau)}u(t\wedge \tau,x)q^{(\ep)}(x)\psi_n(x)dx\right] \\
& \leq\dE \left[\int_{\dR^d}e^{-K_0(t\wedge \tau)}u(t\wedge \tau,x)q^{(\ep)}(x)\psi_n(x)\varrho_m(x)dx\right]  \\
& = \dE\left[\int_0^{t\wedge \tau}\int_{\dR^d}e^{-K_0s}u(s,x)\left[ \cL^*\left(q^{(\ep)}\psi_n\varrho_m\right)(s,x)- K_0q^{(\ep)}(x)\psi_n(x)\varrho_m(x) \right] dxds\right]. \\
\end{aligned} 
\end{equation}
Note that
\begin{equation}
\label{calculation to obtain L1 bound 2}
\begin{aligned}
&\left(q^{(\ep)}(x)\psi_n(x)\varrho_m(x)\right)_{x^i} \\
&\quad= q^{(\ep)}_{x^i}(x)\psi_n(x)\varrho_m(x) + q^{(\ep)}(x)\psi_{nx^i}(x)\varrho_m(x) + q^{(\ep)}(x)\psi_n(x)\varrho_{mx^i}(x),\\
&\left(q^{(\ep)}(x)\psi_n(x)\varrho_m(x)\right)_{x^ix^j} \\
&\quad=q^{(\ep)}_{x^ix^j}(x)\psi_n(x)\varrho_m(x) + q^{(\ep)}(x)\psi_{nx^ix^j}(x)\varrho_m(x) + q^{(\ep)}(x)\psi_n(x)\varrho_{mx^ix^j}(x) \\
&\quad\quad + 2q^{(\ep)}_{x^i}(x)\psi_{nx^j}(x)\varrho_m(x) + 2q^{(\ep)}_{x^i}(x)\psi_n(x)\varrho_{mx^j}(x) + 2q^{(\ep)}(x)\psi_{nx^i}(x)\varrho_{mx^j}(x).
\end{aligned}
\end{equation}

Thus, if we apply \eqref{calculation to obtain L1 bound 2} to the last term of \eqref{calculation to obtain L1 bound 1}, we have
\begin{equation}
\label{calculation to obtain L1 bound 3}
\begin{aligned}
 \cL^* &\left( q^{(\ep)}\psi_n\varrho_m \right)(s,x)  - K_0 \left(q^{(\ep)}\psi_n\varrho_m\right)(x)\\
& = \left( a^{ij}_{x^ix^j} -  b^i_{x^i} + c - K_0\right)\left(q^{(\ep)}\psi_n\varrho_m\right)(x) +  \left( 2 a^{ij}_{x^j} - b^i \right)\left( q^{(\ep)} \psi_n\varrho_m \right)_{x^i}(x) \\
&\quad +  a^{ij} \left( q^{(\ep)}\psi_n\varrho_m \right)_{x^ix^j}(x)\\
&\leq \left[ \left( K - K_0\right)\left( q^{(\ep)}\psi_n \right)(x) + \left( 2a^{ij}_{x^j} - b^i \right)(q^{(\ep)}\psi_n)_{x^i}(x) + a^{ij}(q^{(\ep)}\psi_n)_{x^ix^j}(x)  \right]\varrho_m(x) \\
&\quad + \left[ \left( 2 a^{ij}_{x^j} - b^j \right)\left(q^{(\ep)}\psi_n\right)(x) + 2 a^{ij} \left(q^{(\ep)}\psi_n\right)_{x^j}(x) \right]\varrho_{mx^i}(x)  + a^{ij} \left( q^{(\ep)}\psi_n\varrho_{mx^ix^j} \right)(x).
\end{aligned}
\end{equation}
Observe that on $|x|>R_1>1$, by letting $\ep\downarrow0$,
\begin{equation}
\label{eq:1st derivative of q}
\begin{aligned}
q^{(\ep)}(x) \rightarrow q(x), \quad q^{(\ep)}_{x^i}(x)\rightarrow  q_{x^i}(x)= \left( |x|^{-1} + \delta \right)\frac{x^i}{|x|} q(x),\quad |q_{x^i}(x)|\leq (1+\delta)q(x),
\end{aligned} 
\end{equation} and
\begin{equation}\label{eq:2nd derivative of q}
\begin{aligned}
q^{(\ep)}_{x^ix^j}(x) 
\rightarrow  q_{x^ix^j}(x)
&= \left[ \left( \frac{1}{|x|^2}+\frac{\delta}{|x|} \right)\left( \delta^{ij} - \frac{x^ix^j}{|x|^2} \right) + \left(\frac{2\delta}{|x|}+\delta^2 \right)\frac{x^ix^j}{|x|^2} \right]q(x),  \\
|q_{x^ix^j}(x)|&\leq (2+4\delta+\delta^2)q(x),
\end{aligned} 
\end{equation}	
where $\delta^{ij}$ is the Kronecker delta.
Thus, if we let $n$ be large enough and $\ep\downarrow 0$ in the product rule, by \eqref{eq:1st derivative of q}, and \eqref{eq:2nd derivative of q}, we have that for $|x|>R_1$
\begin{equation}
\label{calculation to obtain L1 bound 4}
\begin{aligned}
\left( K - K_0\right)&q(x)\psi_n(x) + \left( 2a^{ij}_{x^j} - b^i \right)(q(x)\psi_n(x))_{x^i} + a^{ij}(q(x)\psi_n(x))_{x^ix^j}  \\
&=  \left( K - K_0\right)q(x)\psi_n(x) + \left( 2a^{ij}_{x^j} - b^i \right)\left(q_{x^i}(x)\psi_n(x) + q(x)\psi_{nx^i}(x)\right) \\
&\quad+ a^{ij} q_{x^ix^j}(x)\psi_n(x)+2a^{ij}q_{x^i}(x)\psi_{nx^j}(x) + a^{ij}q(x)\psi_{nx^ix^j}(x)  \\ 
&\leq \left[\left(5+6\delta+\delta^2\right)K + N(\delta,K)n^{-1} - K_0\right] q(x)\psi_{2n}(x).\\
\end{aligned}
\end{equation}
To obtain the last inequality, we employ 
\begin{equation*}
|\psi(|x|/n)| + |\psi'(|x|/n)| + |\psi''(|x|/n)|\leq N \psi\left( \frac{|x|}{2n} \right),
\end{equation*}
where $N$ is independent of $n$. Thus, if $n$ large enough, 
\begin{equation}
\label{calculation to obtain L1 bound 4.5}
\begin{aligned}
 \left( K - K_0\right)q(x)\psi_n(x) + \left( 2a^{ij}_{x^j} - b^i \right)(q(x)\psi_n(x))_{x^i} + a^{ij}(q(x)\psi_n(x))_{x^ix^j}  \leq 0.
\end{aligned}
\end{equation}
In addition, observe that
\begin{equation}
\label{calculation to obtain L1 bound 5}
\begin{aligned}
&\limsup_{n\to\infty,\ep\downarrow0}\left[ \left( 2 a^{ij}_{x^j} - b^j \right)\left(q^{(\ep)}\psi_n\right)(x) + 2 a^{ij} \left(q^{(\ep)}\psi_n\right)_{x^j}(x) \right]\varrho_{mx^i}(x)  + a^{ij} \left( q^{(\ep)}\psi_n\varrho_{mx^ix^j} \right)(x)\\
&\quad \leq  \left|  \left( 2 a^{ij}_{x^j} - b^j \right)q(x) + 2 a^{ij} q_{x^i}(x)  \right||\varrho_{mx^j}(x)| + a^{ij} q(x) |\varrho_{mx^ix^j}(x)| \\
&\quad \leq N(d,K)q(x)\left( \left| \varrho_{mx^j}(x) \right| + \left|\varrho_{mx^ix^j}(x)\right| \right).
\end{aligned} 
\end{equation}
Thus, by employing \eqref{calculation to obtain L1 bound 3}, \eqref{calculation to obtain L1 bound 4.5}, and \eqref{calculation to obtain L1 bound 5} to \eqref{calculation to obtain L1 bound 1}, and letting $\ep\downarrow0$ and $n\to\infty$ in order, we have
\begin{equation}
\label{calculation to obtain L1 bound 6}
\begin{aligned}
&\dE \left[\int_{Q_{R_1}}e^{-K_0(t\wedge \tau)}u(t\wedge \tau,x)q(x)dx\right]  \\
&\leq N(d,K)\dE \left[ \int_0^{t\wedge \tau} \int_{ R_1-2m^{-1} < |x| < R_1} e^{-K_0 s} u(s,x)  q(x) \left(\left|\varrho_{mx^j}(x)\right| + \left|\varrho_{mx^ix^j}(x)\right| \right) dxds \right] \\
&\leq N(\delta,d,m,K,R_0)\dE \left[ \int_0^{t\wedge \tau} \int_{ R_1-2m^{-1} < |x| < R_1} e^{-K_0 s} u(s,x) dxds \right].
\end{aligned}
\end{equation}
Therefore, by \eqref{holder reg for ito ineq0}
\begin{equation}
\label{calculation to obtain L1 bound 7}
\begin{aligned}
\dE\left[\int_{Q_{R_1}}q(x)u(t,x)dx\right] 
&\leq N\dE\left[\int_0^{t\wedge \tau}\int_{ R_1-2m^{-1} < |x| < R_1} e^{-K_0s}u(s,x) dxds\right] \\
&\leq N\int_{ R_1-2m^{-1} < |x| < R_1} e^{a|x|} dx, \\
& \leq N,
\end{aligned}
\end{equation}
where $N = N(\delta, a,d,m, H,K,R_0,T)$. Since $Q_R \subset Q_{R_1}$ for any $R>(R_0\vee1)+1$, we have \eqref{L1 bound with weight}. To obtain \eqref{L1 bound with weight - limit}, we employ the dominated convergence theorem with the fact that for $R > R_1$,
\begin{equation*}
	\bi_{Q_{R}}(x)q(x)u(t,x) \leq \bi_{Q_{R_1}}q(x)u(t,x).
\end{equation*}
The lemma is proved.
\end{proof}

\begin{proof}[\textbf{Proof of Lemma~\ref{L_1 bound with weight}}]

Since the method of proof is similar to the proof of Lemma \ref{prop:estimation of quadratic variation}, we only sketch the proof.

For $m,n\in\dN$, choose $\psi_n$ and $\phi_m$ as in the proof of Lemma \ref{prop:estimation of quadratic variation}. Fix $m,n\in\dN$, $\ep>0$, and $\ep_1>0$. Recall that $K_1,K_2$ and $N_1$ in \eqref{ineq:condition_of_K1_and_K2}. Then, by applying It\^o's formula and taking expectation, we have
\begin{equation}
\label{ineq:applying ito's formula 2}
\begin{aligned}
& \dE\left[ \int_{\dR^d} e^{K_2\tau}u(\tau,x)M^{(\ep)}(x)\psi_n(x)\phi_m(x)e^{K_1M^{(\ep_1)}(x)} dx\right] \\
& = \dE\left[\int_0^\tau\int_{\dR^d} e^{K_2s}u(s,x) \cL^* \left(  M^{(\ep)}(x)\psi_n(x) \phi_m(x)e^{K_1M^{(\ep_1)}(x)} \right)  dxds \right]\\
& \quad + \dE \left[K_2 \int_0^\tau\int_{\dR^d} e^{K_2s}u(s,x)M^{(\ep)}(x)\psi_n(x) \phi_m(x)e^{K_1M^{(\ep_1)}(x)}dxds\right].
\end{aligned}
\end{equation}
We note that $K_1$ and $K_2$ are plugged in the exponent of exponential functions instead of $-K_1$ and $-K_2$.
Then, if $\ep_1\downarrow0$ in \eqref{ineq:applying ito's formula 2}, then by \eqref{applying adjoint operator to test function}, \eqref{def of D1ep}, \eqref{eq:convergence of M(x)}, \eqref{eq:derivative of M(x)}, and 
\begin{equation*}
a^{ij}_{x^ix^j}  -  b^i_{x^i} + c + K_1 \left( 2 a^{ij}_{x^j} - b^i \right) M_{x^i} + K_1  a^{ij} M_{x^ix^j} + K_1^2  a^{ij} M_{x^i}M_{x^j} \geq -N_1,
\end{equation*}
we have
\begin{equation}
\label{bar A + bar b}
\begin{aligned}
&\dE\left[\int_{\dR^d} e^{K_2\tau}u(\tau,x)M^{(\ep)}(x)\psi_n(x)\phi_m(x)e^{K_1M(x)} dx \right]
\geq \dE\left[\int_0^\tau e^{K_2 s} \left( \bar {\bf{A}}(m,n,\ep) + \bar {\bf{B}}(m,n,\ep) \right) \, ds\right],
\end{aligned}
\end{equation}
where
\begin{align}
\bar {\bf{A}}(m,n,\ep) &:=  \int_{\dR^d}\left[ 2 a^{ij}_{x^j} - b^i + 2K_1 a^{ij}M_{x^j}(x) \right]  u(s,x)\left(\Phi^{m,n,\ep}(x)\right)_{x^i}e^{K_1M(x)} dx,  \label{def of bar A}\\
\bar {\bf{B}}(m,n,\ep)&:= \int_{\dR^d} u(s,x)  \left(a^{ij}\left(\Phi^{m,n,\ep}(x)\right)_{x^ix^j} +  (K_2 - N_1) \Phi^{m,n,\ep}(x) \right)e^{K_1M(x)} dx  \label{def of bar B}.
\end{align}
The function $\Phi^{m,n,\ep}(x)$ is introduced in \eqref{notation Phi for convenience} and the constant $N_1$ is introduced in \eqref{ineq:condition_of_K1_and_K2}.

Next, similar to \textbf{(Step 2)} and \textbf{(Step 3)} of the proof of Lemma \ref{prop:estimation of quadratic variation}, we obtain
\begin{align}
	\lim_{n\to\infty}\lim_{m\rightarrow \infty}\lim_{\ep \rightarrow 0} \bar {\bf A}(m,n,\ep) 
	&\geq   K^{-1}K_1  \int_{Q_R}u(s,x)  e^{K_1M(x)} dx,
	\label{eq:estimation of bar A}\\
	\lim_{n\to\infty}\lim_{m\rightarrow \infty}\lim_{\ep \rightarrow 0} \bar {\bf B}(m,n,\ep)
	&\geq  -2K \int_{Q_R} u(s,x)  e^{K_1M(x)}dx + K^{-1}\int_{\partial Q_R} u(s,\sigma)d\sigma_R.
	\label{eq:estimation of bar B}
\end{align}
The only difference is to consider $-K$ as a lower bound of the coefficients $a^{ij},b^i$, and $c$.

Then, by combining \eqref{bar A + bar b}, \eqref{eq:estimation of bar A}, and \eqref{eq:estimation of bar B} with taking $K_1\geq 2K^2$, we have
\begin{equation*}
\begin{aligned}
\dE\left[ \int_0^\tau\int_{\partial Q_R} u(s,\sigma)d\sigma_Rds \right]
&\leq K\dE\left[\int_{Q_R} e^{K_2\tau}u(\tau,x)M(x)e^{K_1M(x)} dx\right] \\
&\leq K e^{-K_1R}\dE\left[\int_{Q_R} e^{K_2\tau}u(\tau,x)|x|e^{K_1|x|} dx\right], \\
\end{aligned}
\end{equation*}
and thus
\begin{equation*}
	e^{K_1R}\dE\left[\int_0^\tau\int_{\partial Q_R} u(s,\sigma)d\sigma_Rds \right]
	\leq N(K,T)\dE\left[\int_{Q_R} u(\tau,x)|x|e^{K_1|x|} dx\right].
\end{equation*}
Thus, by \eqref{L1 bound with weight - limit}, we have \eqref{eq:integral bound with weight} for all $\delta = K_1\geq 2K^2$ (see \eqref{ineq:condition_of_K1_and_K2}). The case of $\delta\in(0,2K^2)$ also follows immediately from the above display. The lemma is proved.

\end{proof}


\appendix

\section{Proof of Theorem \ref{thm:existence of solution}}
\label{sec:proof of holder conti of sol}

This section is devoted to proving Theorem \ref{thm:existence of solution}. The proof is based on the $L_p$-theory for stochastic partial differential equations; see \cite{kry1996, kry1999analytic}. The $L_p$-theory enables  us to show the existence of a solution  to \eqref{original_equation} in some stochastic Banach spaces  and also  H\"older continuity of the solution through the Sobolev embedding theorem.  We briefly  introduce  the definitions of stochastic Banach spaces below.  For more information, see \cite{grafakos2009modern,kry1999analytic}. 

\begin{definition}[Bessel potential space]
Let $p>1$ and $\gamma \in \dR$. The space $H_p^\gamma=H_p^\gamma(\dR^d)$ is the set of all tempered distributions $u$ on $\dR^d$ satisfying

$$ \| u \|_{H_p^\gamma} := \left\| (1-\Delta)^{\gamma/2} u\right\|_{L_p} = \left\| \cF^{-1}\left[ \left(1+|\xi|^2\right)^{\gamma/2}\cF(u)(\xi)\right]\right\|_{L_p}<\infty.
$$
Similarly, $H_p^\gamma(\ell_2) = H_p^\gamma(\dR^d;\ell_2)$ is the space of $\ell_2$-valued functions $g=(g^1,g^2,\cdots)$ satisfying
$$ \|g\|_{H_{p}^\gamma(\ell_2)}:= \left\| \left| \left(1-\Delta\right)^{\gamma/2} g\right|_{l_2}\right\|_{L_p} = \left\| \left|\cF^{-1}\left[ \left(1+|\xi|^2\right)^{\gamma/2}\cF(g)(\xi)\right]\right|_{\ell_2} \right\|_{L_p}
< \infty. 
$$
For $\gamma = 0$, we set $H_p^0:= L_p$ and $ H_p^0(\ell_2):=L_p(\ell_2)$.
\end{definition}

Below lemma is well-know result for the H\"older embedding theorem for the Bessel potential spaces.

\begin{lemma}
Let $p>1$ and $\gamma\in\dR$. If $\gamma - d/p = \nu$ for some $\nu\in(0,1)$, then  we have
\begin{equation} 
\label{holder embedding}
\left| u \right|_{C(\dR^d)} + \left[ u\right]_{C^\nu(\dR^d)} \leq N \| u \|_{H_{p}^\gamma},
\end{equation}
where $N = N(\gamma,d,p)$.
\end{lemma}
\begin{proof}
For example, see \cite[Theorem 13.8.1]{krylov2008lectures}.
\end{proof}

\begin{remark} \label{Kernel}
It is well-known that for $\gamma\in (0,\infty)$ and $u\in \cS$, we have
\begin{equation*}
(1-\Delta)^{-\gamma/2}u(x)=\int_{\dR^d}R_{\gamma}(x-y)u(y)dy,
\end{equation*}
where 
\begin{equation*} 
|R_\gamma(x)| \leq N(\gamma,d)\left(e^{-|x|/2}\bi_{|x|\geq2} + A_\gamma(x)\bi_{|x|<2}\right)
\end{equation*}
and
\begin{equation*}
\begin{aligned}
A_{\gamma}(x):=
\begin{cases}
|x|^{\gamma-d} + 1 + O(|x|^{\gamma-d+2}) \quad &\mbox{if} \quad 0<\gamma<d,\\ 
\log(2/|x|) + 1 + O(|x|^{2}) \quad &\mbox{if} \quad \gamma=d,\\ 
1 + O(|x|^{\gamma-d}) \quad &\mbox{if} \quad \gamma>d.
\end{cases}
\end{aligned}
\end{equation*}
Also, we have
$$ (R_{\gamma} \ast R_{\gamma})(x) = R_{2\gamma}(x).
$$
For more information, see \cite[Proposition 1.2.5.]{grafakos2009modern}.

\end{remark}

Recall that $(\Omega, \cF, \dP)$ is a complete probability space with a filtration $\{\cF_t\}_{t\geq0}$ satisfying the usual conditions and $\cP$ is the predictable $\sigma$-field related to $\cF_t$. Next we introduce definitions and properties of stochastic Banach spaces.

\begin{definition}[Stochastic Banach spaces]
For a bounded stopping time $\tau\leq T$, let us denote $\opar0,\tau\cbrk:=\{ (\omega,t):0<t\leq \tau(\omega) \}$.
\begin{enumerate}[(i)]
\item \label{def:Hp,Up}
For $\tau\leq T$, 
\begin{gather*}
\dH_{p}^{\gamma}(\tau) := L_p(\opar0,\tau\cbrk, \mathcal{P}, d\dP \times dt ; H_{p}^\gamma),\\
\dH_{p}^{\gamma}(\tau,\ell_2) := L_p(\opar0,\tau\cbrk,\mathcal{P}, d\dP \times dt;H_{p}^\gamma(\ell_2)),\\
U_{p}^{\gamma} :=  L_p(\Omega,\cF_0, d\dP ; H_{p}^{\gamma-2/p}).
\end{gather*}	
For convenience, we write $\dL_p(\tau):=\dH^{0}_{p}(\tau)$ and $\dL_p(\tau,\ell_2):=\dH^{0}_{p}(\tau,\ell_2)$.

\item \label{def:Hp-gamma}
The norm of each space is defined in the natural way. For example, 
\begin{equation} \label{norm}
\| u \|^p_{\dH^{\gamma}_{p}(\tau)} := \dE \left[\int^{\tau}_0 \| u(t) \|^p_{H^{\gamma}_{p}}dt\right]. 
\end{equation}

\end{enumerate}
\end{definition}

\begin{definition}
\label{def of cH}
Let $\tau\leq T$ be a bounded stopping time and $u \in \dH_p^{\gamma}(\tau)$.
\begin{enumerate}[(i)]
\item 
We write $u\in\cH^{\gamma}_p(\tau)$ if $u_0\in U_{p}^{\gamma}$ and there exists $(f,g)\in
\dH_{p}^{\gamma-2}(\tau)\times\dH_{p}^{\gamma-1}(\tau,\ell_2)$ such that
\begin{equation*}
du = fdt+\sum_{k=1}^{\infty} g^k dw_t^k,\quad   t\in (0, \tau]\,; \quad u(0,\cdot) = u_0
\end{equation*}
in the sense of distributions, i.e., for any $\phi\in C_c^\infty$, the equality
\begin{equation} \label{def_of_sol_2}
(u(t,\cdot),\phi) = (u_0,\phi) + \int_0^t(f(s,\cdot),\phi)ds + \sum_{k=1}^{\infty} \int_0^t(g^k(s,\cdot),\phi)dw_s^k
\end{equation}
holds for all $t\in [0,\tau]$ almost surely.
In this case, we write
\begin{equation*}
\dD u:= f,\quad \dS u:=g.
\end{equation*}

\item
The norm of $\cH_{p}^{\gamma}(\tau)$ is defined by
\begin{equation*}
\| u \|_{\cH_{p}^{\gamma}(\tau)} :=  \| u \|_{\dH_{p}^{\gamma}(\tau)} + \| \dD u \|_{\dH_{p}^{\gamma-2}(\tau)} + \| \dS u \|_{\dH_{p}^{\gamma-1}(\tau,\ell_2)} + \| u(0,\cdot) \|_{U_{p}^{\gamma}}.
\end{equation*}

\end{enumerate}
\end{definition}

Next, we introduce the H\"older embedding theorem for $\cH_p^\gamma(\tau)$.

\begin{theorem}
\label{embedding theorems for cH}
\begin{enumerate}[(i)]
\item  \label{embedding in time}
If $\gamma\in\dR$, $p>2$, $1/2>\beta>\alpha>1/p$, then for any $u\in\cH_p^\gamma(\tau)$, we have $u\in C^{\alpha-1/p}([0,\tau];H_p^{\gamma-2\beta})$ almost surely and 
\begin{equation*}
\dE\| u \|_{C^{\alpha-1/p}([0,\tau];H_p^{\gamma-2\beta})}^p \leq N(\alpha,\beta,d,p,T)\| u \|_{\cH_p^\gamma(\tau)}^p.
\end{equation*}

\item  \label{embedding in time and space}
If $\gamma\in(0,1)$, $p>2$, $\alpha,\beta\in(0,\infty)$ satisfy
\begin{equation*}
\frac{1}{p} < \alpha < \beta < \frac{1}{2}\left( \gamma - \frac{d}{p} \right),
\end{equation*}
then for any $u\in\cH_p^\gamma(\tau)$, we have $u\in C^{\alpha-1/p}([0,\tau];C^{\gamma-2\beta-d/p})$ almost surely and 
\begin{equation*}
\dE\| u \|_{C^{\alpha-1/p}([0,\tau];C^{\gamma-2\beta-d/p})}^p \leq N(\alpha,\beta,d,p,T)\| u \|_{\cH_p^\gamma(\tau)}^p.
\end{equation*}

\end{enumerate}
\end{theorem}
\begin{proof}
For \eqref{embedding in time}, see Theorem 7.2 of \cite{kry1999analytic}. For \eqref{embedding in time and space}, combine \eqref{holder embedding} and \eqref{embedding in time}. In other words, 
\begin{equation*}
\dE\| u \|_{C^{\alpha-1/p}([0,\tau];C^{\gamma-2\beta - d/p}(\dR^d))}^p \leq N\dE\| u \|_{C^{\alpha-1/p}([0,\tau];H_p^{\gamma-2\beta})}^p \leq N\|  u  \|_{\cH_p^{\gamma}(\tau)}.
\end{equation*}
\end{proof}

Note that if $h(t, x, u)$ is globally Lipschitz in $u$, the $L_p$-theory says that there is a unique solution $u \in \cH^{\gamma}_p(\tau)$ for some $\gamma>0$ to \eqref{original_equation} (see \cite[Theorem 6]{ferrante2006spdes} or \cite[Theorem 3.5]{choi2021regularity}). However, since $h(t, x, u)$ is not globally Lipschitz in $u$ (see Assumption \ref{assumptions on h}), we approximate $h$ by a sequence of Lipschitz functions, and then show that the sequence is tight and the limit would be a solution to \eqref{original_equation}. 

Let us take a  symmetric function $\psi\in C_c^\infty(\dR)$ such that $0\leq \psi \leq 1$, $\psi(z) = 1$ on $|z|\leq 1$, $\psi(z) = 0$ on $|z|\geq 2$, and $\sup_{z\in\dR}|\psi'(z)|\leq 1$. {2 Let $n\geq 1$ and} set $\psi_n(z) := \psi(z/n)$. 
Choose $\zeta\in C_c^\infty(\dR)$ such that $\int_{\dR} \zeta dz = 1$ and $\zeta(z) = 0$ if $z<0$ or $z>1$. Define, for all $n\geq 1$,
\begin{equation}
\label{def of hn}
h_n(t,x,u) := n\int_{\dR} h(t,x,z)\zeta\left(n(u-z)\right)dz\psi_n(u) = \int_0^1 h(t,x,u-z/n)\zeta\left(z\right)dz\psi_n(u).
\end{equation}
By Assumption \ref{assumptions on h}, it is easy to see  that  
\begin{equation}
\label{nonnegativity of h_n}
h_n(t,x, 0)  = 0,
\end{equation}
\begin{equation*}
|h_n(t,x,u)| \leq N (1+u),
\end{equation*}
\begin{equation}
\label{assumption of hn:Lipshcitz}
|h_n(t,x,u) - h_n(t,x,v)| \leq N_n |u-v|,
\end{equation}
and
\begin{equation*}
|h_n(t,x,u) - h(t,x,u)|\to0
\end{equation*}
uniformly on compacts as $n\to\infty$. The constant $N_n$ introduced in \eqref{assumption of hn:Lipshcitz} is a positive constant only depending on $n$, and the other constant $N$ does not depend on $n$. 

\vspace{1mm}

The following lemma shows the existence of a \textit{nonnegative} solution to \eqref{original_equation} with $h_n$ instead of $h$ and provides a  uniform estimate of  the moments of $u_n$ in some sense (see \eqref{holder reg of un} and \eqref{bound of u_n}). 
\begin{lemma}
Let $\tau \leq T$ be a bounded stopping time. Suppose that for each $n\in \dN$, $h_n$ is the function defined in \eqref{def of hn} and $\eta\in(0,1]$ is the constant introduced in Assumption \ref{assumption on f}. Then, for any $p\geq2$, there exists $u_n\in\cH_p^\eta(\tau)$ such that $u_n\geq0$, and for any $\phi\in\cS$, the equality
\begin{equation}
\label{cut-off eq}
(u_n(t,\cdot),\phi) = (u_0,\phi)+\int_0^t(u_n(s,\cdot), (\cL^*\phi)(s,\cdot)) ds + \int_0^t\int_{\dR^d}h_n(s,x,u_n(s,x))\phi(x) F(ds,dx)
\end{equation}
holds for all $t\leq \tau$ almost surely, where $\cL^*$ is given  in \eqref{eq:adjoint}. Furthermore, if $\alpha,\beta > 0 $ and $p > 2$ satisfy
\begin{equation*}
\frac{1}{p}<\alpha<\beta<\frac{\eta}{2}-\frac{d}{2p},
\end{equation*}
then $u_n\in C^{\alpha-1/p}([0,\tau];C^{\eta-2\beta-d/p}(\dR^d))$ almost surely. In addition, for $a>0$, we have 
\begin{equation}
\label{holder reg of un}
\sup_{n\geq 1}   \dE\left|  u_n  \Psi_{a/p} \right|^p_{C^{\alpha-\frac{1}{p}}([0,\tau];C^{\eta-2\beta-\frac{d}{p}}(\dR^d) )}\leq  N + N\| u_0 \Psi_{a/p} \|_{U_p^\eta}^p,
\end{equation}
where $\Psi_{a/p}(x) := \frac{1}{\cosh(a|x|/p)}$ and $N = N(\alpha,\beta,\eta,\lambda,a,d,p,K,T)$.
In particular, 
\begin{equation}
\label{bound of u_n}
\sup_{n\geq 1}  \dE\left[\sup_{t\leq \tau,x\in\dR^d}(u_n(t,x))^p e^{-a|x|}\right] \leq  N,
\end{equation}
where $N =N(\eta,\lambda,a,d,p,K,T,u_0)$.
\end{lemma}

\begin{proof}
We first note  that the stochastic integral part in  \eqref{cut-off eq} can be written as  
\begin{equation*}
\int_0^t\int_{\dR^d}h_n(s,x,u_n(s,x))\phi(x) F(ds,dx) =\sum_{k=1}^\infty \int_0^t\int_{\dR^d}h_n(s,x,u_n(s,x))\phi(x) (f \ast e_k)dxdw_s^k,
\end{equation*} 
where $e_k$ and $w_s^k$ are introduced in Remark \ref{random measure}. This allows us to use the $L_p$-theory of SPDEs (\cite{kry1996, kry1999analytic}). Indeed, since $h_n(s,x,u)$ is Lipschitz in $u$, \cite[Theorem 6]{ferrante2006spdes} (or \cite[Theorem 3.5]{choi2021regularity}) yields that there exists a unique solution $u_n\in \cH_p^{\eta}(\tau)$ satisfying \eqref{cut-off eq}. Furthermore, by Theorem \ref{embedding theorems for cH} \eqref{embedding in time and space}, we have $u_n\in C^{\alpha-1/p}([0,\tau];C^{\eta-2\beta-d/p}(\dR^d))$ almost surely.

We now show  non-negativity of $u_n$.  Define  $\bar h_n(t, x, u):= h_n(t, x, u)\bi_{u>0}=h_n(t, x, u_+)$ and consider the following SPDE
\begin{equation}\label{eq:finiteSPDE}
\begin{aligned}
\partial_t \bar u_n(\omega,t,x) & = \cL \bar u_n(\omega,t,x)\,+ \sum_{k=1}^\infty \bar{h}_n(\omega,t,x,\bar u_n(t,x)) (f*e_k)(x) dw_t^k  \quad  (t,x)\in (0,\infty)\times\dR^d
\end{aligned}
\end{equation} 
with initial data $\bar u_n(0,\cdot) = u_0(\cdot)\geq 0.$
Since $\bar{h}_n$ is also Lipschitz in $u$, there exists a unique solution $\bar u_n \in \cH_p^{\eta}(\tau)$ to \eqref{eq:finiteSPDE}. Here, if we show $\bar u_n\geq 0$, then $\bar h_n(t, x, \bar u_n)=h_n(t, x, \bar u_{n+})=h_n(t, x, \bar u_n)$. Thus, by the uniqueness, $\bar u_n=u_n$, which results in the non-negativity of $u_n$.  
Now, for the non-negativity of $\bar u_n$, it just follows from the proof of \cite[Appendix A]{choi2021regularity} (see also \cite[Theorem 1.1]{krylov2007maximum};  the main idea is to consider the finite sum ($\sum_{k=1}^m$) in  the noise part in \eqref{eq:finiteSPDE} whose solutions (say $u_{n, m}$) are nonnegative, thanks to \cite[Theorem 1.1]{krylov2007maximum}, and then show the convergence $u_{n, m} \to u_n$ as $m\to \infty$).

To obtain \eqref{holder reg of un} and \eqref{bound of u_n}, for $a,p>0$, define $\Psi(x) = \Psi_{a/p}(x) = \frac{1}{\cosh(a|x|/p)}$. 
As in  \cite[Lemma 5.5]{choi2021regularity}, we have
\begin{equation}
\label{derivatives of hypercosine}
\begin{aligned}
\Psi_{x^i}(x) \leq \frac{a}{p}\Psi(x)\quad\text{and}\quad\Psi_{x^ix^j}(x) \leq \frac{a^2}{p^2}\Psi(x).
\end{aligned}
\end{equation}
Set
$v_{n}(t,x) := u_n(t,x)\Psi(x)$. Then, for any $\phi \in \cS$, $v_{n}(t,x)$ satisfies
\begin{equation}
\label{eq for v}
\begin{aligned}
(v_{n}(t,\cdot),\phi) &= (\Psi u_0,\phi)+\int_0^t(v_{n}(s,\cdot), (\cL^*\phi)(s,\cdot)) + (g(s,\cdot),\phi) ds \\
&\quad\quad + \int_0^t\int_{\dR^d}h_n(s,x,u_n(s,x))\phi(x) (f\ast e_k)dxdw_s^k,
\end{aligned}
\end{equation}
where $\cL^*$ is in \eqref{eq:adjoint}, and 
\begin{equation}
\label{def of g}
g := u_n( a^{ij}\Psi_{x^ix^j} + 2a^{ij}_{x^i}\Psi_{x^j} - b^i\Psi_{x^i} ) - (2u_na^{ij}\Psi_{x^i})_{x^j}.
\end{equation}
Since $u_n\in \cH_p^{\eta}(\tau)\subset \dH_p^{\eta}(\tau)$, we have
\begin{equation}\label{eq:drifts}
\begin{gathered}
a^{ij}v_{nx^ix^j} + b^{i}v_{
n x^i} + cv_{n} + g \in \dH_p^{\eta-2}(\tau) \quad\text{and}\quad
\Psi u_0 \in L_p(\Omega,\cF_0,H_p^{\eta-2/p}).
\end{gathered}
\end{equation} 
In addition, if we set $\bm{\nu} := (f\ast e_1,f\ast e_2,\dots)$, then we have
\begin{equation}
\label{ineq: bound of the stochastic part}
h_n(u_n)\Psi\bm{\nu} \in \dH_p^{\eta-1}(\tau,\ell_2).
\end{equation}
Indeed, to see \eqref{ineq: bound of the stochastic part}, observe that
\begin{equation*}
\begin{aligned}
&|(1-\Delta)^{-\frac{1-\eta}{2}}(h_n(u_n)\Psi\bm{\nu})(x)|_{\ell_2}^2\\
&\quad=\sum_{k=1}^{\infty}|(1-\Delta)^{-\frac{1-\eta}{2}}(h_n(u_n)\Psi (f\ast e_k))(x)|^2 \\
&\quad=\sum_{k=1}^{\infty}|R_{1-\eta}*(h_n(u_n)\Psi(f*e_k))(x)|^2 \\
&\quad=\sum_{k=1}^{\infty}|(R_{1-\eta}(x-\cdot)h_n(u_n)\Psi,f*e_k)|^2 \\
&\quad =\int_{\dR^d\times\dR^d}R_{1-\eta}(x-(y-z))R_{1-\eta}(x+z)(h_n(u_n)\Psi)(y-z)(h_n(u_n)\Psi)(-z)dzf(dy).
\end{aligned}
\end{equation*} 
Thus, by Minkowski's inequality and H\"older's inequality,
\begin{equation}
\label{l2 calculation}
\begin{aligned}
&\| h_n(u_n)\Psi \bm{\nu} \|^p_{H_p^{\eta-1}(\ell_2)} \\
& = \int_{\dR^d}|(1-\Delta)^{-\frac{1-\eta}{2}}( h_n(u_n)\Psi \bm{\nu})(x)|^p_{\ell_2}dx\\
& = \int_{\dR^d} \left| \int_{\dR^d\times\dR^d}R_{1-\eta}(x-(y-z))R_{1-\eta}(x+z)(h_n(u_n)\Psi  )(y-z)(h_n(u_n)\Psi  )(-z)dz f(dy) \right|^{p/2} dx \\
& = \int_{\dR^d}\left(\int_{\dR^d\times\dR^d}  R_{1-\eta}(y-z)R_{1-\eta}(z)(h_n(u_n)\Psi  )(x+y-z)(h_n(u_n)\Psi  )(x-z)dz f(dy)\right)^{p/2}dx \\
& \leq \left( \int_{\dR^d\times\dR^d}  R_{1-\eta}(y-z)R_{1-\eta}(z)\left(\int_{\dR^d}\left((h_n(u_n)\Psi  )(x+y-z)(h_n(u_n)\Psi  )(x-z)\right)^{p/2}dx\right)^{2/p} dz f(dy) \right)^{p/2} \\
& \leq \| h_n(u_n)\Psi   \|_{L_p}^p\left(\int_{\dR^d}  \int_{\dR^d}R_{1-\eta}(y-z)R_{1-\eta}(z)dz f(dy) \right)^{p/2}\\
&=  \| h_n(u_n)\Psi \|_{L_p}^p \left( \int_{\dR^d}R_{2-2\eta}(y)f(dy) \right)^{p/2}.
\end{aligned}
\end{equation}
Note that 
\begin{equation}
\label{by dalangs condition}
\int_{\dR^d}R_{2-2\eta}(y)f(dy)<\infty
\end{equation} 
by \eqref{reinforced Dalang's condition} and Remark \ref{Kernel}. Therefore,  
\begin{equation}
\label{ineq: bound of the stochastic part1}
\| h_n(u_n)\Psi  \bm{\nu} \|^p_{\dH_p^{\eta-1}(\ell_2,\tau)} 
\leq N\dE\int_0^\tau\int_{\dR^d}| h_n(u_n)\Psi |^p dxdt \leq N\dE\int_0^\tau\int_{\dR^d}|\Psi|^p+| v_{n} |^p dxdt < \infty,
\end{equation}
which implies \eqref{ineq: bound of the stochastic part}. Combining this with \eqref{eq:drifts}, we have  $v_{n}\in\cH_p^{\eta}(\tau)$. Besides, by applying \cite[Theorem 5.1]{kry1999analytic}, for any $t > 0$, we have
\begin{equation}
\label{to obtain estimate 1}
\begin{aligned}
\| v_n \|_{\cH_p^{\eta}(\tau\wedge t)}^p \leq N \left(\| \Psi u_0 \|_{U_p^{\eta}}^p + \| g \|_{\dH_p^{\eta-2}(\tau\wedge t)}^p ds + \| h_n(u_n)\Psi\bm{\nu} \|_{\dH_p^{\eta-1}(\tau\wedge t,\ell_2)}^p \right),
\end{aligned}
\end{equation}
where $N = N(\eta,d,p,K,T)$  and $g$ is given in \eqref{def of g}. Due   to \eqref{derivatives of hypercosine}, for any $s\leq \tau\wedge t$, we have
\begin{equation}
\label{to obtain estimate 2}
\begin{aligned}
\|g(s,\cdot)\|_{H_p^{\eta-2}}^p 
&\leq \| u_n(s,\cdot)( a^{ij}\Psi_{x^ix^j} + 2a^{ij}_{x^i}\Psi_{x^j} - b^i\Psi_{x^i} )\|_{H_p^{\eta-2}}^p + \| (2u_n(s,\cdot)a^{ij}\Psi_{x^i})_{x^j} \|_{H_p^{\eta-2}}^p \\
&\leq N\| v_{n}(s,\cdot) \|_{L_p}^p,
\end{aligned}
\end{equation}
where $N = N(\eta,a,d,p,K)$. In addition, by \eqref{ineq: bound of the stochastic part1},
\begin{equation}
\label{to obtain estimate 3}
\begin{aligned}
&\| h_n(s,\cdot,u_n(s,\cdot))\Psi \bm{\nu} \|_{H_p^{\eta-1}(\ell_2)}^p \leq N \| h_n(s,\cdot,u_n(s,\cdot)) \psi  \|_{L_p}^p \leq N + N\| v_{n}(s,\cdot) \|_{L_p}^p,
\end{aligned}
\end{equation}
where $N = N(\eta,a,d,p,K)$ is independent of $n$. 
We now apply  \eqref{to obtain estimate 2}, \eqref{to obtain estimate 3} to \eqref{to obtain estimate 1}, and use Theorem \ref{embedding theorems for cH} \eqref{embedding in time}, we have
\begin{equation}
\label{to apply gronwall}
\begin{aligned}
\| v_{n} \|_{\cH_p^{\eta}(\tau\wedge t)}^p 
&\leq N + N\|  \Psi u_0\|_{U_p^{\eta}}^p  + N \| v_{n} \|_{\dL_p(\tau\wedge t)}^p  \\
&\leq N + N\|  \Psi u_0\|_{U_p^{\eta}}^p  + N\int_0^{t} \dE\| v_{n} \|_{C([0,\tau\wedge s];L_p)}^p ds \\
&\leq N + N\|  \Psi u_0\|_{U_p^{\eta}}^p  + N\int_0^t \| v_{n} \|_{\cH^{\eta}_p(\tau\wedge s)}^p ds,
\end{aligned}
\end{equation}
where $N = N(\eta,\lambda,a,d,p,K,T)$.
Thus, by the Gronwall's inequality, we have
\begin{equation*}
\begin{aligned}
\| v_{n} \|_{\cH_p^{\eta}(\tau\wedge T)}^p \leq N + N\|  \Psi u_0\|_{U_p^{\eta}}^p,
\end{aligned}
\end{equation*}
where $N = N(\eta,\lambda,a,d,p,K,T)$.  Therefore, by Theorem \ref{embedding theorems for cH} \eqref{embedding in time and space}, we have \eqref{holder reg of un}.
In addition, \eqref{bound of u_n} follows from 
\begin{equation*}
\dE\left[\sup_{t\leq \tau,x\in \dR}(u_n(t,x))^p e^{-a|x|}\right] \leq \dE\left[\sup_{t\leq \tau,x\in \dR^d}(u_n(t,x))^p \frac{1}{\cosh (a|x|)}\right] \leq \dE\left[\sup_{t\leq \tau,x\in \dR^d}|v_{n}(t,x)|^p\right].
\end{equation*}
The lemma is proved.
\end{proof}

We now provide the proof of Theorem \ref{thm:existence of solution}.

\begin{proof}[\bf Proof of Theorem \ref{thm:existence of solution}]\
We first show that there is a nonnegative solution $u$ to \eqref{original_equation} and then show that any solution to \eqref{original_equation} is H\"older continuous almost surely. 

\textbf{(Step 1) (Existence)}
Since  we use the standard compactness argument for the existence of a stochastically weak solution (see e.g.  \cite[Theorem 1.2]{mytnik2006pathwise} and \cite[Theorem 2.6]{shiga1994two}),  we only outline the proof. 

From the bounds \eqref{holder reg of un} and \eqref{bound of u_n}, the Kolmogorov type tightness criterion (see \cite[Lemma 6.3]{shiga1994two} ) implies that $\{ u_n \}_{n\in\dN}$ is tight in $C(\dR_+, C_{tem})$. Then, by Prokhorov's theorem combined with Skorokhod's representation theorem, we have an appropriate probability space and the sequence of solutions $\tilde{u}_n$ on it, which is identical in law to $\{ u_n\}$ and converges to some nonnegative $u$ almost surely in $C(\dR_+, C_{tem})$. Since $h_n$ converges to $h$ uniformly on compact sets and $a^{ij},b^{i},c$ are bounded functions, it can be  shown that the limit $u\in C(\dR_+, C_{tem}) $ is indeed a nonnegative solution of \eqref{eq:sol_int_eq_form}.

\textbf{(Step 2) (H\"older regularity)} We assume $u$ is a solution to \eqref{original_equation} in the sense of Definition \ref{definition_of_solution}. Let $p > \frac{d+2}{\eta}$. For $n\in \dN$ and $a>0$, set
$$ \tau_n :=\tau\wedge n\wedge \inf\left\{ t\geq0: \sup_{x\in\dR^d}u(t,x)e^{-a|x|} \geq n \right\}.
$$
Fix $t\leq \tau_n$ and take $\Psi(x) = \Psi_{2a}(x) = \frac{1}{\cosh(2a|x|)}$. 
Then, since we assume \eqref{condition of h}, as in \eqref{to obtain estimate 3}, we have
\begin{equation*}
\begin{aligned}
\| h(u) \Psi \bm{\nu} \|^p_{\dH_p^{\eta-1}(\ell_2,\tau_n)} 
\leq N \| h(u) \Psi \|_{\dL_p(\tau_n)}^p 
\leq N \int_{\dR^d} \left( 1 + e^{pa|x|} \right)|\Psi(x)|^p dx <\infty,
\end{aligned} 
\end{equation*} 
In addition, if we define
$$ g :=  u\left(a^{ij}\Psi_{x^ix^j}+2a^{ij}_{x^i}\Psi_{x^j}-b^i\Psi_{x^i}\right)  - \left( 2a^{ij} \Psi_{x^i}(x) u \right)_{x^j},
$$
then, similar to \eqref{to obtain estimate 2}, we have $\| g \|_{\dH_p^{\eta-2}(\tau_n)}^p  <\infty.$
Since $\Psi u_0\in L_p(\Omega,\cF_0,H_p^{\eta-2/p})$, \cite[Theorem 5.1]{kry1999analytic} yields that there exists a unique $v\in\cH_p^{\eta}(\tau_n)$ such that
$$\partial_t v = a^{ij}v_{x^ix^j} + b^iv_{x^i} + c v + g +   h(u)\Psi \dot F,\quad t\leq \tau; \quad v(0,\cdot) = \Psi u_0,
$$ 
in the sense of Definition \ref{definition_of_solution} (see also \cite[Remark 5.3]{kry1999analytic}). Note  that $h(u) \Psi$ is used in  the stochastic integral  part, not $h(v) \Psi$. 
Observe that for fixed $\omega\in\Omega$, $\bar u := v - \Psi u$ satisfies
$$ \partial_t \bar u = a^{ij}\bar u_{x^ix^j}+b^i\bar u_{x^i}+c\bar u,\quad 0<t<\tau_n;\quad \bar u(0,\cdot) = 0.
$$
Since $\Psi u, v\in L_p((0,\tau_n)\times\dR^d)$, we have $\bar u\in L_p((0,\tau_n)\times\dR^d)$. Thus, by the deterministic version of $L_p$ theory (e.g. \cite{ladyvzenskaja1988linear}), we have $\bar u(t,\cdot) = 0$ in $L_p(\dR^d)$ for all $t\leq \tau_n$ almost every  $\omega$. Thus, $\Psi u = v\in \cH_p^{\eta}(\tau_n)$. Since $p > \frac{d+2}{\eta}$, by Theorem \ref{embedding theorems for cH} \eqref{embedding in time and space}, for any $\gamma<\eta/2$, we have
$$ \Psi u \in C_{t,x}^{\gamma,2\gamma}([0,\tau_n]\times\dR^d),$$
almost surely. The theorem is proved.



\end{proof}

\vspace{2mm}


\bibliographystyle{plain}

\begin{small}

\noindent\textbf{Beom-Seok Han} [\texttt{hanbeom@postech.ac.kr}]\\
\noindent Pohang University of Science and Technology (POSTECH), Pohang, Gyeongbuk, South Korea \\

\noindent\textbf{Kunwoo Kim} [\texttt{kunwoo@postech.ac.kr}]\\
\noindent Pohang University of Science and Technology (POSTECH), Pohang, Gyeongbuk, South Korea \\

\noindent\textbf{Jaeyun Yi} [\texttt{jaeyun@kias.re.kr}]\\
\noindent Korea Institute for Advanced Study (KIAS), Seoul, South Korea

\end{small}

\end{document}